\documentclass[onefignum,onetabnum]{siamart190516}

\usepackage{lipsum}
\usepackage{amsfonts}
\usepackage{graphicx}
\usepackage{epstopdf}
\usepackage{algorithmic}
\usepackage{color}
\usepackage{cite}
\ifpdf
  \DeclareGraphicsExtensions{.eps,.pdf,.png,.jpg}
\else
  \DeclareGraphicsExtensions{.eps}
\fi

\newsiamremark{remark}{Remark}
\newsiamremark{hypothesis}{Hypothesis}
\crefname{hypothesis}{Hypothesis}{Hypotheses}
\newsiamthm{claim}{Claim}

\newcommand{\ssubset}{\subset\joinrel\subset}

\headers{GUARANTEED A POSTERIORI LOCAL ERROR ESTIMATION}{T. Nakano and X. Liu}

\title{Guaranteed a posteriori local error estimation for finite element solutions of boundary value problems\thanks{2021/12/15.
\funding{The second author is supported by Japan Society for the Promotion of Science: Fund for the Promotion of Joint International Research (Fostering Joint International Research (A)) 20KK0306, 
Grant-in-Aid for Scientific Research (B) 20H01820, 21H00998, and Grant-in-Aid for Scientific Research (C) 18K03411.}}}

\author{Taiga Nakano\thanks{Graduate school of Science and Technology, Niigata University, 8050 Ikarashi 2-no-cho, Nishi-ku, Niigata 950-2181, Japan  (\email{t-nakano@m.sc.niigata-u.ac.jp}).}
\and Xuefeng Liu \thanks{Faculty of Science, Niigata University, 8050 Ikarashi 2-no-cho, Nishi-ku, Niigata 950-2181, Japan  (\email{xfliu@math.sc.niigata-u.ac.jp}) (corresponding author).}
}

\usepackage{amsopn}

\makeatletter
\newcommand*{\addFileDependency}[1]{% argument=file name and extension
  \typeout{(#1)}% latexmk will find this if $recorder=0 (however, in that case, it will ignore #1 if it is a .aux or .pdf file etc and it exists! if it doesn't exist, it will appear in the list of dependents regardless)
  \@addtofilelist{#1}% if you want it to appear in \listfiles, not really necessary and latexmk doesn't use this
  \IfFileExists{#1}{}{\typeout{No file #1.}}% latexmk will find this message if #1 doesn't exist (yet)
}
\makeatother

\ifpdf
\hypersetup{
  pdftitle={Guaranteed a posteriori local error estimation for finite element solutions of boundary value problems},
  pdfauthor={Taiga Nakano, Xuefeng Liu}
}
\fi

\begin{document}
\maketitle

\begin{abstract}{
%Many practical problems occur due to the boundary value problem.
This paper considers the finite element solution of the boundary value problem of Poisson's equation and proposes a guaranteed {\em a posteriori} local error estimation based on the hypercircle method. 
Compared to the existing literature on qualitative error estimation, 
the proposed error estimation provides an explicit and sharp bound for the approximation error in the subdomain of interest, 
and its efficiency can be enhanced by further utilizing a non-uniform mesh. 
 Such a result is applicable to problems without $H^2$-regularity, since it only utilizes the first order derivative of the solution.  
The efficiency of the proposed method is demonstrated by numerical experiments for both convex and non-convex 2D domains  with uniform or non-uniform meshes.
}
\end{abstract}

% REQUIRED

% REQUIRED
\section{Introduction}
\label{sec:1}
\medskip
This paper studies the finite element solution of the boundary value problem (BVP)
of Poisson's equation, and proposes an  {\em a posteriori} local error estimation method for finite element solution, based on the hypercircle method, i.e., the Prager--Synge theorem.

The motivation of this research originates from the error analysis of the four-probe method, which has been used for the resistivity measurement of semiconductors over the past century \cite{Miccoli-2015}. The image of the four-probe method is illustrated in Figure \ref{fig:four-probe}: four probes $A, B, C$, and $D$ are aligned on the surface of the sample; a constant current $I_{AD}$ is applied between $A$ and $D$ and the potential difference $V_{BC}$ between $B$ and $C$ is measured. 
The resistivity $\rho$ is then calculated by $\rho=F_c V_{BC}/I_{AD}$, where $F_c$ is the correction factor.
As an important quantity for high-precision measurement, $F_c$ is evaluated theoretically by considering the governing equation of the distribution of the potential. 
A well-used model for the potential distribution $u$ is described by the following boundary value problem of 
Poisson's equation (see, e.g., \cite{Miccoli-2015,Yamashita-1984}):
\begin{eqnarray}
   -\Delta u = 2 \rho \: I_{AD} \: (\delta (A;x) - \delta (D;x))  \mbox{ in } \Omega;\quad\frac{\partial u}{\partial \mathbf{n}}=0 \mbox{ on }\partial\Omega\:,
\label{eq:resist_model}
\end{eqnarray}
where $\delta(A;x)$ and $\delta(D;x)$ are Dirac's delta functions located at $A$ and $D$, respectively. 
Note that this model regards the current $I_{AD}$ as a point charge on the surface of the sample.
By setting $\rho \: I_{AD}=1$, the value of $F_c$ can be evaluated by 
$$
F_c = \frac{1}{u(B) - u(C)}\:.
$$
The calculation of $F_c$ only utilizes the potential $u$ at the probes $B$ and $C$, i.e., the local information of the solution around the probes.
To have a sharp estimation of 
the correction factor $F_c$, the local error around the probes is of interest and the local error estimation for the FEM approximation to $u$ is wanted. 
Note that the right-hand side of \eqref{eq:resist_model} does not belong to the $L^2$ space. 
In this study, instead of the equation \eqref{eq:resist_model}, a model problem $-\Delta u=f$ with $f\in L^2(\Omega)$ is considered.  The technique to solve equation \eqref{eq:resist_model} directly will be discussed in our subsequent papers.

\begin{figure}[h!]
\begin{center}
\includegraphics[width=4.0cm]{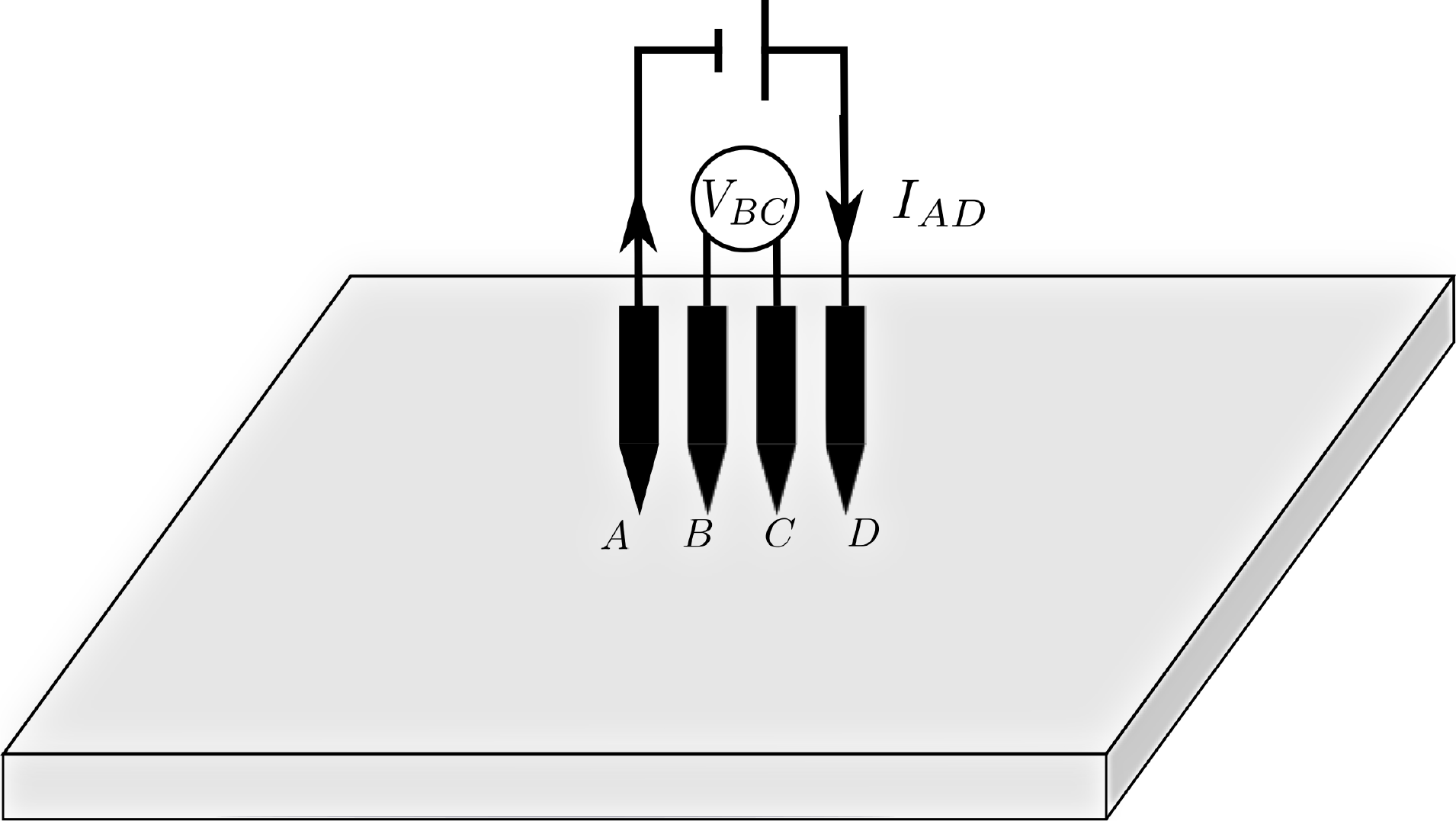} 
\caption{The four-probe method}
\label{fig:four-probe}
\end{center}
\end{figure}

There is some literature related on local error estimation for finite element solutions (see, e.g.,\cite{Nitsche-1974, Wahlbin-1991, Xu-2000, Liao-2003,Demlow-2011, Demlow-2010, Demlow-2016, Wahlbin-1995, Schatz-1988, Schatz-1995}).
The local error estimation was first studied by Nitche and Schatz in \cite{Nitsche-1974}. 
In the studies of \cite{Nitsche-1974,Wahlbin-1991}, for the subdomain $\Omega_0$ of interest, an intermediate subdomain $\Omega_1$ such that $\Omega_0 \ssubset \Omega_1 \ssubset \Omega$ is utilized to deduce the following error estimation: for $u \in H^l(\Omega)$,
$$
 \| u-u_h \|_{s,\Omega_0} \leq C(h^{l-s} \| u\|_{l,\Omega_1} + \|u-u_h\|_{-p,\Omega_1})\:, 
$$
for $s = 0$ or $1$ and  $p$ as a fixed integer. 
 In \cite{Nitsche-1974} , an estimation for quasi-uniform meshes was provided. In \cite{Xu-2000,Xu-2001}, based on the knowledge of the  local error distribution, Xu and Zhou purposed a parallel technique that uses coarse grid to approximate the low frequencies part of the residual error and then uses fine grid for the high frequency part. Discussions on relaxing the assumption in \cite{Nitsche-1974} applied to the mesh can be found in \cite{Demlow-2011, Demlow-2007}. In the field of adaptive finite element methods, error indicators based on the local error estimation have been well studied; see, e.g.,\cite{ Xu-2000, Xu-2001, Liao-2003, Demlow-2010, Demlow-2016}.

The above results on local error estimation mainly focus on the qualitative analysis (e.g., convergence rate) of local error terms,  while the explicit bound for the local error estimation is not available.

In this paper, we propose a quantitative error estimation method for the local error of the finite element solutions. 
Such a method is regarded as an extension of the explicit error estimation theorem developed by Liu  \cite{Liu-2013}, which inherits the idea of Kikuchi \cite{Kikuchi-2007-Rem.A.Post.Err.Est} to utilize the hypercircle method. The idea of local error estimation was also introduced in a concise manner in our previous work \cite{Nakano-2019} published in Japanese. Here, thorough discussions along with detailed numerical examples are provided to describe this newly developed local error estimation method. 

\medskip

The application of the hypercircle method to the {\em a posteriori} error estimation can also be found in  \cite{Braess-2007, Neittaanmaki-2004}.
Instead of developing $\mathbf{p}_h\in H(\mbox{div\:};\Omega)$ by processing the discontinuity of $\nabla u_h$ across the edges of elements \cite{Braess-2007, Neittaanmaki-2004}, the feature of Kikuchi's approach and our method is to construct the hypercircle by utilizing $\mathbf{p}_h \in H(\mbox{div\:};\Omega)$ such that $\mbox{div\:} {\mathbf{p}}_h +f_h=0$ holds exactly, where $f_h$ is the projection of $f$ to totally discontinuous piecewise polynomials.

The hypercircle method, namely the Prager--Synge theorem, was developed more than fifty years ago by \cite{Prager-1947} for elastic analysis. Around the same time \cite{Prager-1947}, based on the $T^*T$ theory of Kato \cite{Kato-1953}, Fujita developed a method similar to the hypercircle method \cite{Fujita-1955}, which has been applied to develop the point-wise estimation method for boundary value problems with specially constructed base functions. The extension of Kato--Fujita's approach to finite element method for local error estimation will be considered in our succeeding work.

\medskip

The local error estimation proposed in this paper has the following features.
\begin{itemize}
\item [(1)] As a quantitative result, it provides explicit estimation for the energy error in the subdomain of interest. 

\item [(2)] The method deals with domains of general shapes in the natural way and is applicable to non-convex domains where a singularity may appear around the re-entry corner of the boundary. 

\item[(3)] There are no constraints on the mesh generation of the domain, compared to the stringent requirement of a uniform mesh in past studies on local error estimation.
Also, the proposed local error estimator has an optimal convergence rate for the finite element solutions when non-uniform meshes are utilized to obtain finer solution approximation over the subdomain of interest.

\end{itemize}

\medskip

The rest of the paper is organized as follows:
In section \ref{sec:2}, we provide preliminary information on the problem setting and basic knowledge about the finite element method spaces.
In section \ref{sec:3}, the global error estimation based on the hypercircle method developed in \cite{Liu-2013} is introduced.
In section \ref{sec:4}, the details of the local error estimation are described. 
In section \ref{sec:5}, the qualitative analysis about the convergence rate of the proposed local error estimation is discussed.
In section \ref{sec:6}, the numerical examples for Poisson's equation over the square domain and the L-shaped domain are presented. 
Finally, in section \ref{sec:7}, we summarize the conclusions and discuss future studies.

\section{Preliminary}
\label{sec:2}
\subsection{Problem settings}
Throughout this study, the domain $\Omega$ is assumed to be a bounded polygonal domain of $\mathbb{R}^2$. Thus, $\Omega$ can be completely triangulated without any gap near the boundary. Standard symbols are used for the Sobolev spaces $H^m(\Omega)(m>0)$.
The norm of $L^2(\Omega)$ is written as $\| \cdot \|_{L^2(\Omega)}$ or $\| \cdot \|_{\Omega}$ .
Symbols $|\cdot|_{H^m(\Omega)},\| \cdot  \|_{H^m(\Omega)}$ denote semi-norm and norm of $H^m(\Omega)$, respectively.
Let $(\cdot,\cdot)$ be the inner product of $L^2(\Omega)$ or $(L^2(\Omega))^2$.
Sobolev space $W^{1,\infty}(\Omega)$ is a function space where weak derivatives up to the first order are essentially bounded on $\Omega$.
The standard vector valued function space $H(\mbox{div};\Omega)$ is defined as follows:
$$H(\mbox{div};\Omega):= \left \{ \mathbf{q} \in (L^2(\Omega))^2 ;~  \mbox{div } \mathbf{q} \in L^2(\Omega) \right \}.$$

In this paper, the finite element solution for the following model boundary value problem will be discussed:
\begin{eqnarray}
\label{eq:model}
   -\Delta u = f  \mbox{ in } \Omega ,\quad\displaystyle\frac{\partial u }{\partial\mathbf{n}} = g_N\mbox{ on } \Gamma_N,\quad u= g_D\mbox{ on } \Gamma_D.
\end{eqnarray}
Here, $\Gamma_N$ and $\Gamma_D$ are disjoint subsets of $\partial \Omega$ satisfying $\Gamma_N \cup \Gamma_D = \partial \Omega$; $\mathbf{n}$ is the unit outer normal direction on the boundary and $\frac{\partial }{\partial \mathbf{n}}$ is the directional derivative along $\mathbf{n}$ on $\partial \Omega$.

Let $S$ be a subdomain of $\Omega$ of interest. 
Suppose that $u_h$ is an approximate solution to the problem \eqref{eq:model}. The error of $(\nabla u - \nabla u_h)$ in the subdomain $S$ will be evaluated in this study.

\medskip

The weak form for the aforementioned problem is given by:
\begin{eqnarray}
\mbox{Find } u \in V \mbox{ s.t. } (\nabla u,\nabla v) = (f,v) + (g_N,v)_{\Gamma_N} 
   , \quad\forall v \in V_0 .\label{eq:model_weak}
\end{eqnarray}
  where
$$(g_N,v)_{\Gamma_N} := \int_{\Gamma_N} g_N v ~ds .$$
In case $\Gamma_D$ is not an empty set, the function space $V$ of the trial function and the function space $V_0$ of the test function are defined by
\begin{equation*}
  \label{eq:v0}
  V := \left \{v \in H^1(\Omega); v =g_D \mbox{ on } \Gamma_D \right \},~V_0 := \left \{v \in H^1(\Omega); v =0 \mbox{ on } \Gamma_D \right \} . 
\end{equation*}
For an empty $\Gamma_D$, the definition of $V$ and $V_0$ are modified as follows:
\begin{equation*}
  \label{eq:v0_new}
  V = V_0 = \left \{v \in H^1(\Omega); \int_\Omega v ~dx =0  \right \}.
\end{equation*}

\subsection{Finite element space setting}
To prepare for the discussion on the newly developed local error estimation in \S \ref{sec:4}, we review the standard FEM approaches to \eqref{eq:model}.
To simplify the discussion, assume $g_D,g_N$ in the boundary conditions of the model problem \eqref{eq:model} to be piecewise linear and piecewise constant at the boundary edges of $\mathcal{T}^h$, respectively.
Let $\mathcal{T}_h$ be a proper triangulation of the domain $\Omega$. Given an element $K \in \mathcal{T}_h$, let $h_K$ denote the length of longest edge of $K$.
The mesh size $h$ of $\mathcal{T}_h$ is defined as follows: 
$$h := \max_{K \in \mathcal{T}_h} h_K .$$
On each element $K \in \mathcal{T}_h$, the set of polynomials with degree up to $d$ is denoted by $P_d(K)$. 
Let $V_h,V_{h,0}$ denote the finite element spaces consisting of piecewise linear and continuous functions,  the boundary conditions of which follow the settings of $V$ and $V_0$, respectively.
The conforming finite element formulation of \eqref{eq:model_weak} is given by
\begin{eqnarray}
\mbox{Find } u_h \in V_h \mbox{ s.t. } (\nabla u_h,\nabla v_h) = (f,v_h) + (g_N,v_h)_{\Gamma_N}, \quad\forall v_h \in V_{h,0}. \label{eq:CF}
\end{eqnarray}

To provide the local error estimation for $(\nabla u - \nabla u_h)$ over the subdomain $S$, let us introduce the following finite element spaces.
\begin{itemize}
\item[(a)] Piecewise constant function space:
$$X_h := \left \{v_h \in L^2(\Omega) ~:~ v_h|_K \in P_0(K),~ \forall K \in \mathcal{T}_h  \right \} .$$
In case $\Gamma_D$ is empty, it is further required that $\int_\Omega v dx =0$ for $v_h \in X_h$.

%is defined as follows:
%    $$X_h := \left \{v \in L^2(\Omega) ~:~ v_h|_K \in P_0(K),~ \forall K \in \mathcal{T}_h , \int_\Omega v dx =0 ~ \right \} .$$
\item[(b)] The Raviart--Thomas finite element space:
$$RT_h := \left \{\mathbf{p}_h \in H(\mbox{div};\Omega) ~:~ \mathbf{p}_h|_K =(a_K+c_Kx,b_K+c_Ky) \mbox{ for } K \in \mathcal{T}_h \right \} .$$
$$RT_{h,0} =\{ \mathbf{p}_h \in RT_h ~:~ \mathbf{p}_h \cdot \mathbf{n} =0 \mbox{ on } \Gamma_N \}\;.
    $$
    Here, $a_K,b_K,c_K \in P_0(K)$ for $K \in \mathcal{T}_h$.
\end{itemize}

  The standard mixed finite element formulation of \eqref{eq:model} reads:
Find $ (\mathbf{p}_h,\mu_h) \in RT_h \times X_h,~ \mathbf{p}_h \cdot \mathbf{n} = g_N \mbox{ on } \Gamma_N$, s.t.
\begin{equation}
     \label{eq:MF}
       (\mathbf{p}_h,\mathbf{q}_h)  +  (\mbox{div } \mathbf{q}_h,\mu_h)  +
       (\mbox{div } \mathbf{p}_h,\eta_h)   =    \int_{\Gamma_D} g_D (\mathbf{q}_h \cdot \mathbf{n}) ~ds  -(f,\eta_h)
\;,
\end{equation}
for $(\mathbf{q}_h ,\eta_h) \in  RT_{h,0} \times X_h$.

\medskip

Define the projection $\pi_{h}:L^2(\Omega) \to X_h $ such that for $f \in L^2(\Omega)$, 
\begin{equation*}
   (f-\pi_{h} f,\eta_h) =0, \quad \forall \eta_h \in X_h .
   \label{eq:projection}
\end{equation*}
The following error estimation holds for $\pi_h$,
\begin{equation}
   \| f- \pi_h f\|_{\Omega} \leq C_0 h |f|_{H^1(\Omega)}, \quad \forall f \in H^1(\Omega). \label{eq:inter_err}
\end{equation}  

To give a concrete value of $C_0$ in \eqref{eq:inter_err}, let us define $C_0(K)$ as a constant that depends on the shape of the triangle $K \in \mathcal{T}_h$ and satisfies 
$$   \| f- \pi_h f\|_{K} \leq C_0(K) |f|_{H^1(K)}, \quad \forall f \in H^1(K). $$
By using $C_0(K)$, the constant $C_0$ that depends on triangulation can be defined by 
\begin{equation}
 C_0:=\max_{K \in \mathcal{T}_h} \frac{C_0(K)}{h}. \label{eq:global-L2-inter-err}
\end{equation}
The previous studies \cite{Kikuchi-2006,Kikuchi-2007-Est.Intpol.Err.Const,Laugesen-2010} reported that the optimal value of $C_0(K)$ is given by $C_0(K):=h_K/j_{1,1} (\le 0.261 h_K) $ using positive minimum root $j_{1,1}\approx 3.83171$ of  the first kind Bessel's function $J_1$.

\section{Global {\em a priori} error estimation for the finite element solutions}
\label{sec:3}

In this section, we introduce the global error estimation developed in \cite{Liu-2013, PVNC-2018, Li-2018}, which will be used in Theorem \ref{theorem:local_est} for local error estimation.
We focus on the global {\em a priori} error estimation for problems with homogeneous boundary value conditions, which fits the needs in the proof for Theorem \ref{theorem:local_est}.
For global {\em a priori} error estimation of non-homogeneous boundary value problems, refer to \cite{Li-2018}.

\medskip

As a preparation for Theorem \ref{theorem:local_est}, let us consider the following boundary value problem.
\begin{eqnarray}
   -\Delta\phi = f  \mbox{ in } \Omega ,\quad\displaystyle\frac{\partial\phi }{\partial\mathbf{n}} = 0 \mbox{ on } \Gamma_N,\quad\phi= 0 \mbox{ on } \Gamma_D.
\label{eq:homo_model}
\end{eqnarray}
The weak formulation of \eqref{eq:homo_model} seeks $\phi \in V_0$, s.t., 
\begin{equation}
    (\nabla \phi,\nabla v) = (f,v),  \quad \forall v \in V_0 .\label{eq:homo_model_weak}
\end{equation}
  The Galerkin projection operator $P_h:V_{0} \to V_{h,0}$ satisfies,   for $ v \in V_{0}$
\begin{equation}
  \label{eq:galerkin}
   (\nabla(v-P_h v),\nabla v_h) =0 , \quad \forall v_h \in V_{h,0}\:.
\end{equation}  

  In \cite{Liu-2013}, the following quantity $\kappa_h$ is introduced for the purpose of {\em a priori} error estimation to the Galerkin projection $P_h \phi$:
$$ 
 \kappa_h := 
 \max_{f_h \in X_h} ~~ 
 \min_{ \substack{v_h \in V_{h,0} ,~ {\mathbf{q}}_h \in RT_{h,0}, \\   \text{div } {\mathbf{q}}_h + f_h =0} }  \frac{\| \nabla v_h -{\mathbf{q}}_h\|}{\|f_h\|} .
 $$

The theorem below provides an {\em a priori} error estimation using $\kappa_h$ and the Prager--Synge theorem.
\begin{theorem}[Global {\em a priori} error estimation \cite{Liu-2013}]
\label{theorem:global_est}
  Given $f \in L^2(\Omega)$, let $\phi$ be the solution to  \eqref{eq:homo_model_weak}.
  Then, the following error estimation holds.
\begin{eqnarray} 
&& |\phi - P_h \phi|_{H^1(\Omega)}   \leq   C(h) \|f\|_{\Omega}, \label{eq:grobal-H1}\\
&&\|\phi - P_h \phi\|_{\Omega}  \leq  C(h) |\phi - P_h  \phi|_{H^1(\Omega)} \label{eq:grobal-L2}\le  C(h)^2 \|f\|_{\Omega}\:,
\end{eqnarray}  
  where $C(h):=\sqrt{\kappa_h^2+(C_0h)^2}$; $C_0$ is the quantity defined in \eqref{eq:global-L2-inter-err}.

\end{theorem}

\begin{remark}
When the exact solution $\phi$ belongs $H^2(\Omega)$, the constant $C(h)$ appearing in \eqref{eq:grobal-H1}, \eqref{eq:grobal-L2} can be replaced for the error constant of the Lagrange interpolation.
For example, in \S \ref{subsec:square}, it is possible to use $C_h=0.493h$ as $C(h)$ for the right-angled triangle mesh.
\end{remark}
\begin{remark}
Calculation of $\kappa_h$: given $f_h \in X_h$, let $R_h:X_h \to V_h$, $T_h:X_h \to RT_h$ be the linear operators that map $f_h$ to the Lagrange FEM approximation of $\nabla \phi$ and the Raviart-Thomas FEM approximation of $\nabla  \phi$, respectively.
Then, $\kappa_h$ is characterized by the following maximum formulation:
$$ 
  \kappa_h = \max_{f_h \in X_h} \frac{\|(R_h - T_h) f_h\|_{\Omega} }{\|f_h\|_{\Omega}}\:,
  $$
which is determined by a matrix eigenvalue problem. See \cite{Liu-2013,Li-2018} for a more detailed discussion on the calculation of $\kappa_h$.
\label{rem:kappa}
\end{remark}

\section{Weighted hypercircle formula and main result}
\label{sec:4}
In this section, we propose {\em a posteriori } local error estimation for the finite element solutions.
Let $S(\subset \Omega)$ be the subdomain of interest. 
In \S \ref{subsec:weight_function}, the weighted inner product and weighted norm corresponding to $S$ will be introduced through a cutoff function $\alpha$.
In \S \ref{subsec:weighted_hypercircle}, we show the weighted hypercircle formula as an extension of Theorem \ref{theorem:PS}.  
The result of the local error estimation will be provided in \S \ref{subsec:main}.

\subsection{The weight function}
\label{subsec:weight_function}
 Let $\Omega'$ be the extended domain of $S$ with a band of width $\varepsilon$, that is, $\Omega':= \left \{ x \in \Omega ~;~ \mbox{dist}(x, S) < \varepsilon  \right \}$. Denote the band surrounding $S$ by $B_S$. Refer to Figure \ref{fig:alpha-graph}-(a),(b) for two examples of $S$ and $B_S$.
The weight function $\alpha \in W^{1,\infty}(\Omega)$ is defined as a piecewise polynomial with the following property.
\begin{equation*}
     \alpha(x,y)  =
    \left\{ 
    \begin{array}{l}      1 \quad (x,y) \in S \\      0 \quad (x,y) \in (\Omega')^c     \end{array} 
    \right. , \quad  0 \leq \alpha(x,y) \leq 1, \mbox{ for } (x,y) \in B_S\:.
\end{equation*}

\begin{figure}[h]
\begin{center}
\includegraphics[width=3.5cm]{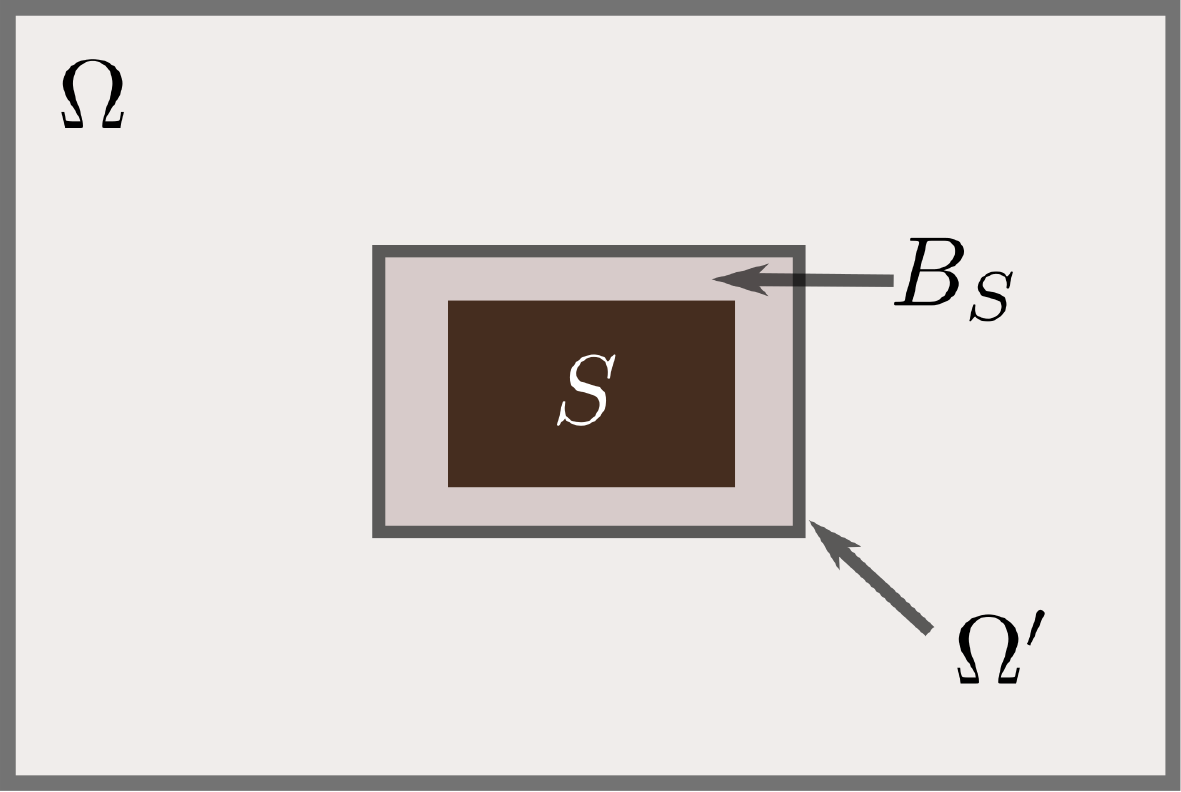} \quad
\includegraphics[width=3.5cm]{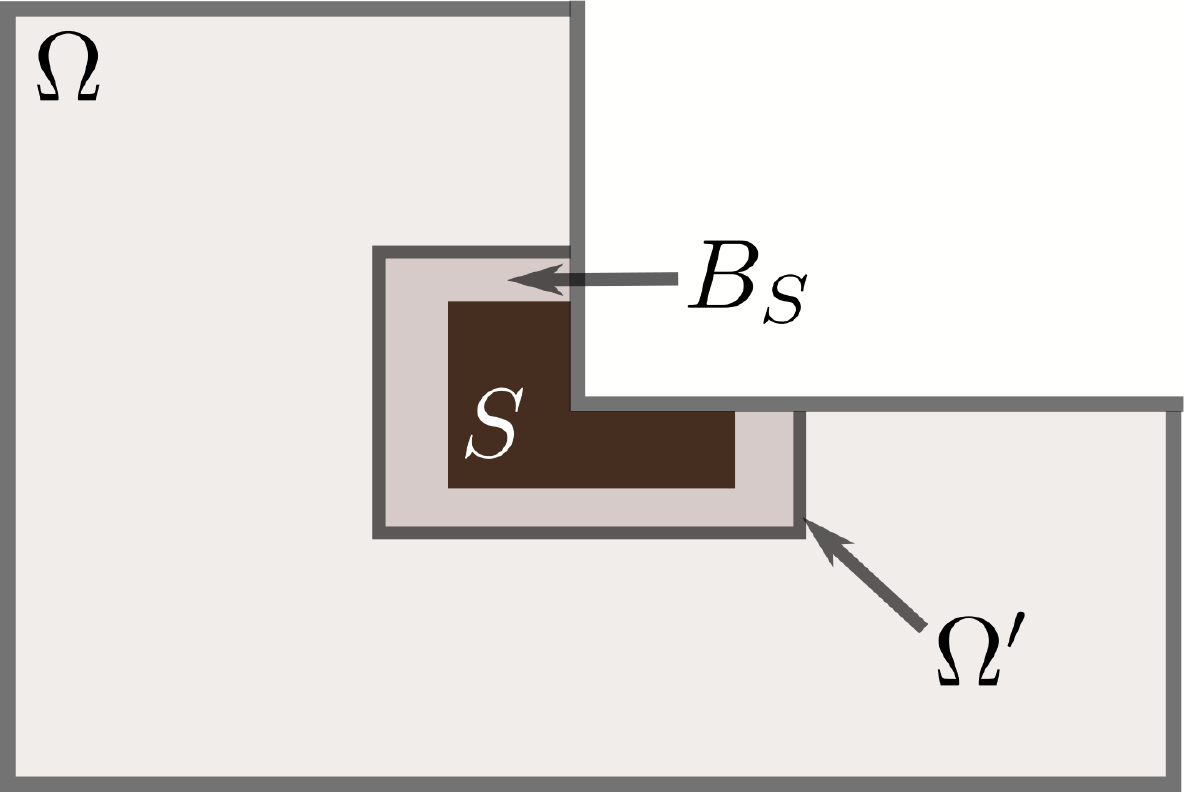} \quad
\includegraphics[width=4.2cm]{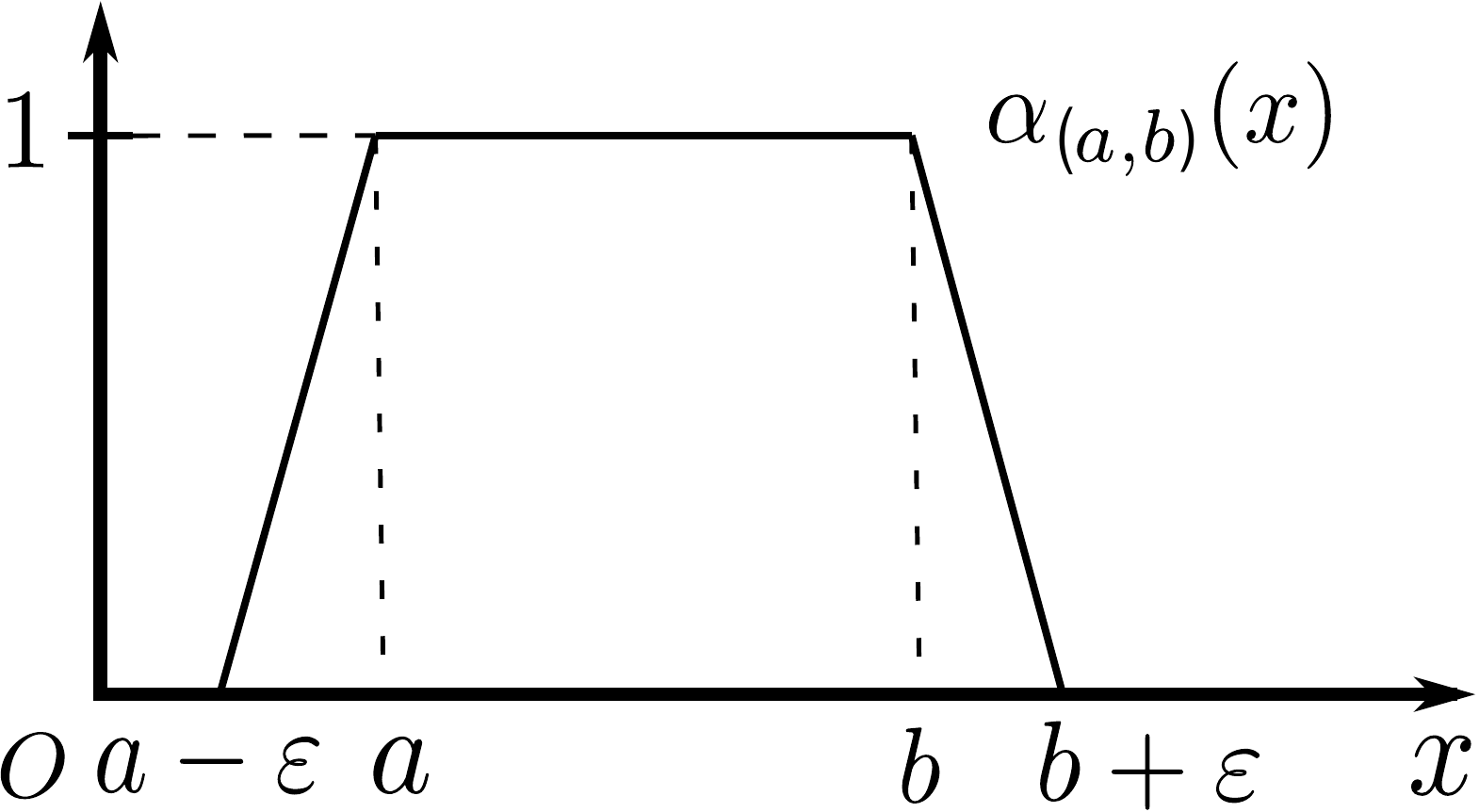}  \\
    {(a) Square domain.} \quad\quad\quad {(b) L-shaped domain.} \quad\quad  {(c) Graph of $\alpha_{(a,b)}(x).$} \quad
\caption{Definition of the weight function $\alpha$.}
\label{fig:alpha-graph}
\end{center}
\end{figure}

To construct a concrete weight function $\alpha(x,y)$, let us define $\alpha_{(a,b)}$ over interval $(a,b)$ as follows.
\begin{eqnarray*}
\alpha_{(a,b)}(x)  :=
\left\{ 
\begin{array}{ll}
   1+ (x-a)/\varepsilon\quad& x \in (a-\varepsilon,a]  \\
   1 \quad& x \in(a,b)  \\
   1- (x-b)/\varepsilon\quad& x \in [b,b+\varepsilon)  \\
   0 \quad&\mbox{otherwise}
\end{array} 
\right.\:.
\end{eqnarray*}
Refer to Figure \ref{fig:alpha-graph}-(c) for the graph of $\alpha_{(a,b)}$.
For $S$ being a rectangular subdomain constructed by the Cartesian product of two open intervals $(x_a, x_b)$, $(y_a, y_b)$, the weight function $\alpha$ can be defined by $\alpha(x,y) =\min{\{\alpha_{(x_a, x_b)}(x), \alpha_{(y_a,y_b)}(y)\} }$.

\medskip

The weighted inner product and norm are defined by using $\alpha$ as follows.
\begin{itemize} 
\item[(a)] Weighted inner product $(\cdot,\cdot)_{\alpha}$: ~ For $f,g \in L^2(\Omega)$ or $f,g \in (L^2(\Omega))^2$,
\begin{equation*}
  (f,g)_{\alpha} := \int_{\Omega'}  \alpha f \cdot g ~dx.
\end{equation*}

\item[(b)] Weighted norm $\| \cdot \|_{\alpha}:~$  For $f \in L^2(\Omega)$, 
\begin{equation*}
  \| f\|_{\alpha} : = \sqrt{(f,f)_{\alpha}}= \sqrt{\int_{\Omega'} f^2 \alpha ~dx} \quad (=\left \| f\sqrt{\alpha} \right \|_{\Omega}). 
\end{equation*}
\end{itemize}  
The following inequalities hold.
\begin{equation} \| f \|_S \leq \| f \|_{\alpha} \leq \| f \| _{\Omega'} \leq \| f \| _{\Omega}. \label{eq:norm_relation} \end{equation}

\subsection{Weighted hypercircle formula}
\label{subsec:weighted_hypercircle}
In this sub-section, a weighted hypercircle formula is proposed, which can be regarded as an extention to the classical Prager--Synge theorem below. 
\begin{theorem}[Prager--Synge's theorem\cite{Prager-1947}]\label{theorem:PS}
 Let $\phi$ be the solution of \eqref{eq:homo_model}.
 For any $v \in V_0$ and $\widetilde{\mathbf{p}} \in H(\mbox{\em div};\Omega)$ satisfying
$$
 \mbox{\em{div} }\widetilde{\mathbf{p}}+f=0 ,~ \widetilde{\mathbf{p}} \cdot \mathbf{n} = 0 \mbox{ on }  \Gamma_N \:, 
 $$
 we have, 
\begin{equation}
 \|\nabla \phi - \nabla v \|_{\Omega}^2 + \| \nabla \phi- \widetilde{\mathbf{p}}\|_{\Omega}^2 = \|\nabla v - \widetilde{\mathbf{p}} \|_{\Omega}^2. \label{eq:Hypercircle}
\end{equation}
\end{theorem}

For weighted norms introduced in the previous section, we have the following extended formulation of the hypercircle \eqref{eq:Hypercircle}.
\begin{theorem}
\label{theorem:alpha-PS}
 Let $u$ be the solution of \eqref{eq:model}.
 For any $v \in V_0$ and $\mathbf{p} \in H(\mbox{\em div};\Omega)$ satisfying
$$ \mbox{\em div }\mathbf{p}+f=0 \mbox{ in } \Omega, \quad  \mathbf{p} \cdot \mathbf{n} = g_N \mbox{ on }  \Gamma_N .$$
Then,  
\begin{equation*}
     \|\nabla u - \nabla v \|_{\alpha}^2 + \| \nabla u- \mathbf{p}\|_{\alpha}^2 \leq \|\nabla v - \mathbf{p} \|_{\alpha}^2  + 2 \sqrt{2} \|\nabla \alpha\|_{L^{\infty}(\Omega)} \|u-v \|_{\Omega'} \|\nabla u -    \mathbf{p} \|_{\Omega'} .
    %\label{eq:alpha-Hypercircle} 
\end{equation*}
\end{theorem}
\begin{proof}
 The expansion of $\|\nabla v - \mathbf{p} \|_{\alpha}^2=\|(\nabla v - \nabla u )+ (\nabla u- \mathbf{p}) \|_{\alpha}^2$ tells that
\begin{equation}
 \|\nabla v - \mathbf{p} \|_{\alpha}^2  =   \|\nabla v - \nabla u \|_{\alpha}^2 + \| \nabla u- \mathbf{p}\|_{\alpha}^2  + 2 (\nabla v - \nabla u ,\nabla u- \mathbf{p})_{\alpha}. 
 \label{eq:alpha-expansion}
\end{equation}
Let $w: = v- u$. 
Below, we show the estimation for the cross-term of \eqref{eq:alpha-expansion}, i.e., 
$(\nabla w, \nabla u - \mathbf{p})_{\alpha} $.

To deal with $(\nabla w ,\nabla u)_{\alpha} $, let us take the test function as $\alpha w$ in \eqref{eq:model_weak} and apply the chain rule to $\alpha w$, i.e.,
$
\nabla(\alpha w) = w \nabla \alpha +
\alpha \nabla w
 $.
 Then we have
\begin{eqnarray}
    (\nabla w, \nabla u)_{\alpha}    = -\int_{\Omega }  w \nabla\alpha\cdot\nabla u  ~dx +  \int_{\Omega} f \cdot (\alpha w)  ~dx  +\int_{\Gamma_N} g_N \cdot (\alpha w) ds.  \label{eq:alpha-uw}
\end{eqnarray}
 For $(\nabla w ,\mathbf{p})_{\alpha}$, Green's formula tells that
\begin{eqnarray}
 (\nabla w,\mathbf{p})_{\alpha} & = &\int_{\Omega} \nabla w \cdot (\alpha\mathbf{p})  ~dx = \int_{\Gamma_N} g_N\cdot (\alpha w) ds - \int_{\Omega} w \mbox{ div} (\alpha\mathbf{p}) ~dx\nonumber\\
&=&\int_{\Gamma_N} g_N\cdot (\alpha w) ds- \int_{\Omega} (\alpha w)  \mbox{ div } \mathbf{p} ~dx - \int_{\Omega} w  \nabla\alpha\cdot\mathbf{p} ~dx.  \nonumber\\
&=&  - \int_{\Omega} w \nabla\alpha\cdot\mathbf{p} ~dx + \int_{\Omega} (\alpha w) f ~dx +  \int_{\Gamma_N} g_N\cdot (\alpha w) ds.  
\label{eq:alpha-pw}
\end{eqnarray} 
 By taking \eqref{eq:alpha-uw}-\eqref{eq:alpha-pw} and noticing that $\alpha = 0$ on $\Omega \setminus \Omega'$, we have
\begin{equation*}
   (\nabla v - \nabla u ,\nabla u- \mathbf{p})_{\alpha} = - \int_{\Omega} (v-u) \nabla \alpha \cdot ( \nabla u- \mathbf{p} )~dx = - \int_{\Omega'} (v-u) \nabla \alpha \cdot ( \nabla u- \mathbf{p} )~dx. \label{eq:cross-term}
\end{equation*}
By applying H\"older's inequality, we have
\begin{equation}
 \label{eq:local_inequality_alpha}
   |(\nabla v - \nabla u ,\nabla u- \mathbf{p})_{\alpha}| \leq \sqrt{2} \| \nabla \alpha \|_{L^{\infty}(\Omega)} \| u-v\|_{\Omega'} \|\nabla u -\mathbf{p} \|_{\Omega'}.
\end{equation}
From \eqref{eq:alpha-expansion} and \eqref{eq:local_inequality_alpha}, we draw the conclusion.
\end{proof}

\begin{remark}
  Theorem \ref{theorem:alpha-PS} holds no matter $\partial S \cap \partial \Omega = \emptyset$ or not, as can be confirmed in the proof.
\end{remark}

\subsection{{\em A posteriori} local error estimation for finite element solutions}
\label{subsec:main}
 As a preparation to the argument of the main result, let us follow the idea of Kikuchi \cite{Kikuchi-2007-Rem.A.Post.Err.Est} to introduce auxiliary functions $\overline{u} \in V$ and  $ \overline{u}_h \in V_h$ as the solutions to the following equations.
\begin{eqnarray}
     (\nabla \overline{u},\nabla v )& = &(\pi_h f,v) + (g_N,v)_{\Gamma_N}, ~\forall v \in V_0\:; \label{eq:aux} \\
   (\nabla \overline{u}_h,\nabla v_h ) &=& (\pi_h f,v_h) + (g_N,v_h)_{\Gamma_N}, ~\forall v_h \in V_{h,0}\:. \label{eq:aux_CF}
\end{eqnarray}
Both the functions are introduced only for error analysis in a theoretical way, and the above equations do not need to be solved explicitly.
\begin{lemma}
\label{Lem:aux}
  For $\overline{u}$ of \eqref{eq:aux} and $\overline{u}_h$ of \eqref{eq:aux_CF}, the following estimations hold.
\begin{eqnarray} 
   |u-\overline{u}|_{H^1(\Omega)} \leq C_0h \|f-\pi_hf\|_{\Omega} ,\label{eq:second}\\
   |u_h - \overline{u}_h |_{H^1(\Omega)} \leq C_0 h \|f-\pi_h f\|_{\Omega}. \label{eq:third}
\end{eqnarray}
Here, $C_0$ is the projection error constant defined in \eqref{eq:global-L2-inter-err}.
\end{lemma}

\begin{proof}
According to the definitions of $u$ and $\overline{u}$,
$$
(\nabla(u-\overline{u}),\nabla v) = (f-\pi_h f,v) = (f-\pi_h f,v-\pi_h v), \quad v \in V_0.
$$
By taking $v = u - \overline{u}$ in the above equation and using the error estimation of projection $\pi_h$ in \eqref{eq:inter_err}, we have
$$|u-\overline{u}|_{H^1(\Omega)} ^2 = |(\nabla(u-\overline{u}),\nabla (u-\overline{u}))| \leq \|f-\pi_hf \|_{\Omega} \cdot C_0h |u-\overline{u}|_{H^1(\Omega)}.$$
  Hence,
$$ |u-\overline{u}|_{H^1(\Omega)} \leq C_0h \|f-\pi_hf \|_{\Omega} .$$
Estimation \eqref{eq:third} is obtained in the same way.
\end{proof}

  With the selected $\pi_h f \in X_h$, the property $\mbox{div}(RT_h) = X_h$ makes it possible to find $\mathbf{p}_h \in RT_h$ such that
\begin{equation}
  \label{eq:cond-ph}
  \mbox{div }\mathbf{p}_h+\pi_h f=0,~ \mathbf{p}_h \cdot \mathbf{n} = g_N \mbox{ on } \Gamma_N\:.
\end{equation}
  In this case, the following hypercircle equation holds: 
\begin{equation}
      \label{eq:fem-hypercircle}
      \|\nabla \overline{u} - \nabla v_h \|^2_{\Omega} +
      \|\nabla \overline{u} - \mathbf{p}_h\|^2_{\Omega} =\|\nabla v_h - \mathbf{p}_h \|^2_{\Omega}\:.
\end{equation}

\begin{remark} The estimation \eqref{eq:second} along with the hypercircle \eqref{eq:fem-hypercircle} leads to an {\em a posteriori } estimation of the global error.
\begin{equation}
      \label{eq:global-error-estimation}
      \|\nabla (u-u_h)\|_\Omega \le \|\nabla u_h - \mathbf{p}_h \|_\Omega  + C_0h\|f-\pi_h f\|_\Omega\:.
\end{equation}
  Here, $\mathbf{p}_h$ can be chosen freely to approximate $\nabla u$ under the condition \eqref{eq:cond-ph}.

\end{remark}

  Below, we apply Theorem \ref{theorem:alpha-PS} to the current function settings.

\begin{lemma}
\label{Lem:aux-hypercircle}
    Let $\overline{u}$ and $\overline{u}_h$ be the solutions of \eqref{eq:aux} and \eqref{eq:aux_CF}, respectively.
    For $\mathbf{p}_h \in RT_h$ satisfying \eqref{eq:cond-ph}, the following estimation holds,
$$
    \|\nabla (\overline{u} -  \overline{u}_h)\|_{\alpha}^2  \leq \|\nabla  \overline{u}_h - \mathbf{p}_h\|_\alpha^2 
    +2\sqrt{2} \|\nabla \alpha\|_{L^{\infty}(\Omega)} \|\overline{u}- \overline{u}_h\|_{\Omega} ~ \|\nabla \overline{u} - \mathbf{p}_h\|_{\Omega} .
  $$
\end{lemma}
\begin{proof}
   Take $f := \pi_h f ,~\mathbf{p} := \mathbf{p}_h,~ u :=\overline{u} $ and $ v := \overline{u}_h$ in Theorem \ref{theorem:alpha-PS}, then the following inequality holds.
$$
    \|\nabla (\overline{u} -  \overline{u}_h)\|_{\alpha}^2 + \|\nabla \overline{u} - \mathbf{p}_h \|_\alpha^2  \leq \|\nabla  \overline{u}_h - \mathbf{p}_h\|_\alpha^2 
    +2\sqrt{2} \|\nabla \alpha\|_{L^{\infty}(\Omega)} \|\overline{u}- \overline{u}_h\|_{\Omega'} ~ \|\nabla \overline{u} - \mathbf{p}_h\|_{\Omega' } .
  $$
Because $\Omega'\subset \Omega$, we conclude by replacing $\|\cdot\|_{\Omega'}$ with $\|\cdot\|_{\Omega}$.
\end{proof}

\medskip

  To state the results in Theorem \ref{theorem:local_est} and \ref{theorem:local_est2}, let us define the following four computable quantities $\overline{E}_1, ~\overline{E}_2, ~E_1$ and $E_2$.
\begin{eqnarray*}
&&\overline{E}_1 :=  \|\nabla\overline{u}_h -\mathbf{p}_h\|_{\alpha}, \quad
E_1 :=  \|\nabla u_h - \mathbf{p}_h\|_{\alpha} + C_0h\|f-\pi_h f\|_\Omega,\\
&&\overline{E}_2 := \left\{ 2\sqrt{2} C(h)\cdot\|\nabla\alpha\|_{L^{\infty}(\Omega)}  \right\} ^{1/2}\cdot\|\nabla\overline{u}_h - \mathbf{p}_h\|_{\Omega} , \\
&&E_2 := \left\{ 2\sqrt{2} C(h)\cdot\|\nabla \alpha\|_{L^{\infty}(\Omega)} \right\}^{1/2}  \cdot\|\nabla u_h - \mathbf{p}_h\|_{\Omega}\:. ~ \\
%&&(\overline{E}_2 := \left\{ 2\sqrt{2} C(h)\cdot\|\nabla \alpha\|_{L^{\infty}(\Omega)}  \cdot\|\nabla \overline{u}_h - \mathbf{p}_h\|_{\Omega} \cdot\|\nabla \overline{w}_h - \mathbf{p}_h\|_{\Omega}\:\right\}^{1/2} .) ~ \\
%&&(E_2 := \left\{ 2\sqrt{2} C(h)\cdot\|\nabla \alpha\|_{L^{\infty}(\Omega)}  \cdot\|\nabla u_h - \mathbf{p}_h\|_{\Omega} \cdot\|\nabla \overline{w}_h - \mathbf{p}_h\|_{\Omega}\:\right\}^{1/2} .) ~ \\
\end{eqnarray*}

\medskip

\begin{theorem}[{\em A posteriori} local error estimation for $\overline{u}$]
\label{theorem:local_est}
  Let $u$ and $\overline{u}_h$ be the solutions of \eqref{eq:model_weak}, \eqref{eq:aux_CF}, respectively.
  For $\mathbf{p}_h \in RT_h$ satisfying
$$ \mbox{\em{div} }\mathbf{p}_h+\pi_h f=0,~ \mathbf{p}_h \cdot \mathbf{n} = g_N \mbox{ on }  \Gamma_N ~, $$
the following local error estimation holds. 
\begin{equation}
    \label{eq:local_est}
    \|\nabla u - \nabla \overline{u}_h \|_{S}  \leq  \sqrt{\overline{E}_1^2+\overline{E}_2^2} + C_0h \|f -  \pi_h f\|_{\Omega}. 
\end{equation}
\end{theorem} 

\begin{proof}
   With $\overline{u}$ defined in \eqref{eq:aux} and the triangle inequality, we obtain:
\begin{equation*}
  \| \nabla (u-\overline{u}_h) \|_S \leq \| \nabla (u-\overline{u}) \|_S + \| \nabla (\overline{u} - \overline{u}_h) \|_S \leq \| \nabla (u-\overline{u}) \|_{\Omega}+ \| \nabla (\overline{u} - \overline{u}_h) \|_{\alpha}. 
  %\label{eq:expansion2} 
\end{equation*}
 By applying the estimation of $\|\nabla(u-\overline{u})\|_\Omega$ in Lemma \ref{Lem:aux} and the estimation of $\|\nabla(\overline{u}-\overline{u}_h)\|_\alpha$ in Lemma \ref{Lem:aux-hypercircle}, the following estimation holds.
\begin{eqnarray}
\|\nabla (u-\overline{u}_h) \|_S
&\leq&
\left\{ \|\nabla\overline{u}_h - \mathbf{p}_h\|_\alpha^2 +2 \sqrt{2} \|\nabla\alpha\|_{\infty} \|\overline{u}-\overline{u}_h\|_{\Omega} ~ \|\nabla\overline{u} - \mathbf{p}_h\|_{\Omega} \right\}^{\frac{1}{2}}  \notag\\
&& + C_0 h \| f - \pi_h f \|_{\Omega} \:   .
\label{eq:origin_main}
\end{eqnarray}
 Next, we give the estimation for $\|\overline{u}-\overline{u}_h\|_{\Omega},~\|\nabla \overline{u} - \mathbf{p}_h\|_{\Omega} $ in \eqref{eq:origin_main}.

\medskip

(a) From Theorem \ref{theorem:PS}, the hypercircle below is available for $\overline{u}$ defined in \eqref{eq:aux},
\begin{equation}
\label{eq:hypercircle-in-proof}    
\|\nabla \overline{u} - \nabla v_h \|_{\Omega}^2 + \| \nabla \overline{u}- \mathbf{p}_h\|_{\Omega}^2 = \|\nabla v_h - \mathbf{p}_h \|_{\Omega}^2, \quad  \forall v_h \in V_h \:.
\end{equation}
By taking $v_h:=\overline{u}_h$,
  we obtain the estimation of $\|\nabla \overline{u} - \mathbf{p}_h\|_{\Omega} $:
\begin{equation}
    \|\nabla\overline{u}-\mathbf{p}_h \|_{\Omega} \leq \|\nabla\overline{u}_h - \mathbf{p}_h \|_{\Omega}\:. \label{eq:up-est}
\end{equation}

(b) To give the estimation of $\|\overline{u}-\overline{u}_h\|_{\Omega}$, let us define the dual problem.
$$\mbox{Find } \phi \in V_0 \mbox{ s.t. } (\nabla \phi,\nabla v) = (\overline{u}-\overline{u}_h,v), \quad \forall v \in V_0. $$
    By applying $P_h$ defined in \eqref{eq:galerkin} along with the {\em a priori } estimation \eqref{eq:grobal-H1} in Theorem \ref{theorem:global_est}, we have,
\begin{eqnarray*}
\|\overline{u}-\overline{u}_h\|_{\Omega}^2 %&= &
&\leq&
\|\nabla (\phi-P_h\phi) \|_{\Omega} \cdot\|\nabla(\overline{u}-\overline{u}_h) \|_{\Omega} .
\\
&\le&  C(h)  \: \|\overline{u}-\overline{u}_h\|_{\Omega} \cdot \|\nabla(\overline{u}-\overline{u}_h) \|_{\Omega} 
. 
\end{eqnarray*}
Notice that \eqref{eq:hypercircle-in-proof} with $v_h:=\overline{u}_h$ implies 
\begin{equation*}
    \|\nabla\overline{u} - \nabla\overline{u}_h \|_{\Omega} \leq \|\nabla\overline{u}_h - \mathbf{p}_h \|_{\Omega}\:. \label{eq:overu-v}
\end{equation*}
Thus, we have the estimation of $ \| \overline{u}-\overline{u}_h \|_{\Omega}  $:
\begin{equation}
    \label{eq:u-bar-minus-u-h-bar}
    \| \overline{u}-\overline{u}_h \|_{\Omega}  \leq  
    C(h) \|\nabla (\overline{u} - \overline{u}_h) \| _{\Omega}
    \le 
    C(h) \|\nabla\overline{u}_h - \mathbf{p}_h \|_{\Omega}.
\end{equation}

  Apply \eqref{eq:up-est}  and \eqref{eq:u-bar-minus-u-h-bar} to the first term of the right-hand side of \eqref{eq:origin_main}, 
\begin{eqnarray*}
&\quad&\|\nabla\overline{u}_h - \mathbf{p}_h\|_\alpha^2 +2\sqrt{2} \:\|\nabla\alpha\|_{\infty} \|\overline{u}-\overline{u}_h\|_{\Omega} ~ \|\nabla\overline{u} - \mathbf{p}_h\|_{\Omega}  \\
&&\leq\|\nabla\overline{u}_h - \mathbf{p}_h\|_\alpha^2  +   2\sqrt{2} \:\|\nabla\alpha\|_{\infty} \:
   C(h)\: \|\nabla\overline{u}_h - \mathbf{p}_h\| _{\Omega} ^2  \\
%&&( = \|\nabla\overline{u}_h - \mathbf{p}_h\|_\alpha^2  + \sqrt{2} \:\|\nabla\alpha\|_{\infty} \: C(h)\: \|\nabla\overline{u}_h - \mathbf{p}_h\| _{\Omega} \: \|\nabla \overline{w}_h - \mathbf{p}_h\| _{\Omega} )
%\\
&& = {\overline{E}_1^2 + \overline{E}_2^2}.
\end{eqnarray*}
Now, we draw the conclusion by sorting the estimation of \eqref{eq:origin_main}. 
\end{proof}

\begin{theorem}
\label{theorem:local_est2}
 Under the assumptions of Theorem \ref{theorem:local_est}, the following estimation holds.
\begin{equation} \|\nabla u - \nabla u_h \|_{S}   \leq \sqrt{E_1^2 + E_2^2} + 2 C_0h \|f -  \pi_h f\|_{\Omega} \:.  
  \label{eq:local_est2}
\end{equation}
\end{theorem}

\begin{proof}
First, we apply the triangle inequality to $(u-\overline{u}_h) + (\overline{u}_h-{u}_h )$ and the estimation \eqref{eq:third} in Lemma \ref{Lem:aux} to have,
$$
  \|\nabla (u - {u}_h) \|_S 
  \le 
  \|\nabla (u_h  - \overline{u}_h ) \|_S  
  +
  \|\nabla ( u - \overline{u}_h ) \|_S  
  \leq C_0h \|f-\pi_h f\|_{\Omega}  + 
  \| \nabla ({u} - \overline{u}_h) \|_{S}\:.
 $$

   Next, we apply the result in Theorem \ref{theorem:local_est} to $\| \nabla ({u} - \overline{u}_h) \|_{S}$ and process the term $\overline{u}_h$ in $\overline{E}_1$ and $\overline{E}_2$.  
Because $\overline{u}_h$ is the best approximation to $\overline{u}$ in $V_h$, the hypercircle \eqref{eq:fem-hypercircle} with respect to $\pi_h f$ leads to
$$
\| \nabla \overline{u}_h - \mathbf{p}_h  \|_\Omega
\le \| \nabla u_h - \mathbf{p}_h  \|_\Omega \:.
$$
For the term $\| \nabla \overline{u}_h-\mathbf{p}_h \|_{\alpha}$  in $\overline{E}_2$ , apply the triangle inequality and \eqref{eq:third} to obtain
$$
\| \nabla \overline{u}_h-\mathbf{p}_h \|_{\alpha} \leq  \| \nabla (\overline{u}_h-u_h) \|_{\alpha} + \| \nabla u_h - \mathbf{p}_h \|_{\alpha} 
 \leq C_0h \|f-\pi_h f\|_{\Omega}  +\| \nabla u_h - \mathbf{p}_h \|_{\alpha}\: .
$$
Now, we can conclude as in \eqref{eq:local_est2}.

\end{proof}

\section{Convergence analysis and application to non-uniform mesh}
\label{sec:5}

In this section, we have an analysis on the convergence behavior for the proposed {\em a posteriori} error estimation and show its application in efficient computing with non-uniform meshes. 

For a solution $u \in H^2$  solved by FEM over a uniform mesh with mesh size $h$, the global error terms $\|\nabla u_h - \mathbf{p}_h\|_{\Omega}$ and $C(h)$ have the convergence rate as $O(h)$.  Thus, the following convergence rate is available in Theorem \ref{theorem:local_est2} .
\begin{equation}
\label{eq:meshsize_depedency}
 E_1= O(h), ~ E_2 = O(h^{1.5}); ~~ \|\nabla u - \nabla u_h \|_{S}  \le O(h) 
\end{equation}
Note that the order of $E_2$ can be further improved by selecting the Raviart-Thomas FEM space with  higher degree and theoretically $E_2$ can be arbitrarily small by selecting a good approximation of $\mathbf{p}_h$ to $\overline{u}$ using an independent quite refined mesh. Below is a detailed argument for this property.

{\noindent \bf Improving convergence rate of the global error term.} Let $V_{h}^{k} , 
~RT_{h}^{k-1}$ $(k\ge 1)$ be the conforming finite element space of degree $k$ and the Raviart-Thomas FEM space of degree $k-1$, respectively.
    Suppose $\overline{u}$ defined in \eqref{eq:aux} has $H^2$-regularity. Let us consider the following FEM solutions corresponding to $\overline{u}$:
    \begin{itemize}
        \item $\overline{w}_h \in V_{h}^{k}$ as the best approximation to  $\overline{u}$ under $|\cdot|_1$ norm ;
        \item $\widetilde{\mathbf{p}}_h \in RT_{h}^{k-1}$ as the best approximation  to $\nabla \overline{u}$ under $L^2$ norm, subject to the condition \eqref{eq:cond-ph}.
    \end{itemize}
    Then the following hypercircle equation holds.
    $$
    \|\nabla \overline{u} - \nabla \overline{w}_h \|_{\Omega}^2 + \| \nabla \overline{u}- \widetilde{\mathbf{p}}_h\|_{\Omega}^2 = \|\nabla \overline{w}_h - \widetilde{\mathbf{p}}_h \|_{\Omega}^2 \:.
    $$
    Since for a smooth enough solution $\overline{u}$, both $\|\nabla \overline{u} - \nabla \overline{w}_h \|_{\Omega}$ and $\| \nabla \overline{u}-\widetilde{\mathbf{p}}_h \|_{\Omega}$ have the convergence rate as $O(h^k)$, we have 
     \begin{equation*}
     \label{eq:up-est-higher}
        ( \|\nabla\overline{u}-\widetilde{\mathbf{p}}_h\|_{\Omega}  \leq ) \|\nabla\overline{w}_h - \widetilde{\mathbf{p}}_h \|_{\Omega} = O(h^k)~~ (k\ge 2)\:. 
    \end{equation*}
    By replacing the estimation of \eqref{eq:up-est} with $\|\nabla\overline{u}-\widetilde{\mathbf{p}}_h \|_{\Omega} \leq \|\nabla\overline{w}_h - \widetilde{\mathbf{p}}_h\|_{\Omega}$, 
    the global error terms $E_2, \overline{E}_2$ in Theorem \ref{theorem:local_est} and  \ref{theorem:local_est2}  become

    \begin{eqnarray}
    &&\overline{E}_2 := \left\{ 2\sqrt{2} C(h)\cdot\|\nabla \alpha\|_{L^{\infty}(\Omega)}  \cdot\|\nabla \overline{u}_h - \widetilde{\mathbf{p}}_h\|_{\Omega} \cdot\|\nabla \overline{w}_h - \widetilde{\mathbf{p}}_h\|_{\Omega}\:\right\}^{1/2} , ~ \nonumber \\
    &&E_2 := \left\{ 2\sqrt{2} C(h)\cdot\|\nabla \alpha\|_{L^{\infty}(\Omega)}  \cdot\|\nabla u_h - \widetilde{\mathbf{p}}_h\|_{\Omega} \cdot\|\nabla \overline{w}_h - \widetilde{\mathbf{p}}_h\|_{\Omega}\:\right\}^{1/2} . ~ \nonumber 
    \end{eqnarray}

    Note that both $\|\nabla \overline{u}_h - \widetilde{\mathbf{p}}_h\|_{\Omega} $ and $ \|\nabla u_h - \widetilde{\mathbf{p}}_h\|_{\Omega}$ have the convergence rate as $O(h)$, which can be confirmed by utilizing the hypercircle involving $\overline{u}, \overline{u}_h$ and $\mathbf{p}_h$.
    %Thus, the convergence rate of terms $E_2, \overline{E}_2$ is possible to be $O(h^2)$ for $u \in H^2(\Omega)$ and even higher if $u$ is smooth enough. 
    Therefore we have the convergence rate  as
    \begin{equation}
    \label{eq:improved_global_order}
        E_2, \overline{E}_2 =O(h^{(1+k/2)}) \:(k \geq 1). 
    \end{equation}
    It is worth to point out the selection of $\overline{w}_h$ and $\mathbf{p}_h$ is independent on the  objective FEM solution $u_h$. Thus, a large value of $k(\ge 2)$ will result in an improved convergence rate for global error terms; see numerical results in Figure \ref{fig:mesh_dependency_improved}, \ref{fig:local_mesh_dependency_imporoved}  of \S \ref{subsec:square}.
{
{\bf \noindent Application to non-uniform mesh.} Theoretical analysis on local error estimation tells that for $u \in H^2(\Omega)$ (\cite{Xu-2000,Xu-2001}): 
\begin{equation}
\label{eq:typcal_local_err_est}
\|u-  u_h\|_{1,S} = O(h_{\Omega'} + h_G^2). %\le  C ( \| \nabla (u - u_h) \|_{S} + \|u-u_h\|_{\Omega})  
\end{equation}
Here, $h_{\Omega'}$ denotes the mesh size of $\Omega'$; $h_G$ the one for the mesh outside of $\Omega'$. Estimation \eqref{eq:typcal_local_err_est} implies an asymptotically optimal error with rate $O(h_{\Omega'})$ in $H^1$ norm  locally by taking $h_G = O(\sqrt{h_{\Omega'}})$.

Such {\em a priori} estimation motivates us to apply our proposed {\em a posteriori} error estimator to non-uniform meshes to have more efficient computation and error estimation. 
Here, the local error estimator in Theorem \ref{theorem:local_est2} is denoted by $$\widehat{E}_{L} :=  \sqrt{E_1^2+E_2^2 } + 2C_0h_G \|f - \pi_h f\|_{\Omega}\: .$$
For $f \in H^1(\Omega)$, the convergence rate of the projection error term $C_0h\|f-\pi_h f\|$ in $E_1$ is $O(h_G^2)$. 
For the error term $E_1$, it is expected that the local term $\|\nabla u_h -\mathbf{p}_h \|_{\alpha}= O(h_{\Omega'} + h_{G}^2)$. However, such 
a result is not discussed yet in the existing literature. 
In this paper, rather than theoretical analysis, by numerical experiment in 
\S \ref{subsec:square}, we confirm the asymptotic behavior of $E_1=O(h_{\Omega'})$  over a non-uniform mesh when the mesh size is selected as $h_G = \sqrt{h_{\Omega'}}$. 
With analogous argument to the one for \eqref{eq:improved_global_order},  it is easy to see that the 
convergence rate of the global error term $E_2$ and $\overline{E}_2$ are  dominated by the global mesh size $h_G$:
    \begin{equation}
    \label{eq:improved_global_order_non_uniform_mesh}
        E_2, \overline{E}_2 =O(h_G^{(1+k/2)}) \:(k \geq 1). 
    \end{equation}

In \S 6.2, by numerical examples, we investigate the the convergence rate of error estimators when the mesh size is chosen as $h_G = \sqrt{h_{\Omega'}}$. In case that only $u_h \in V_h^{(1)}$ and $\mathbf{p}_h \in RT_h^{(0)}$ are used in the local error estimation, it is validated that
$$
 E_1= O(h_{\Omega'}), ~ E_2 = O(h_G^{1.5}); ~~ \widehat{E}_L \le O(h_{\Omega'} + h_G^{1.5}).
$$
In case that $\overline{w}_h \in V_h^{(k)}$ and $\widetilde{\mathbf{p}}_h \in RT_h^{(k-1)}~ (k \ge 2)$ along with \eqref{eq:improved_global_order_non_uniform_mesh} are also used in the local error estimation, we have
$$
 E_1= O(h_{\Omega'}), ~ E_2 = O(h_G^{(1+k/2)}); ~~ \widehat{E}_L \le O(h_{\Omega'} + h_G^{(1+k/2)}).
$$

\section{Numerical experiments}
\label{sec:6}
\subsection{Preparation}

\indent{} The selection of bandwidth of the $B_{S}$ is important in the local error estimation. A large bandwidth of $B_{S}$ leads to a large value of $E_1$, while a small bandwidth of $B_{S}$ results in a large value of $\| \nabla \alpha \|_{L^{\infty}(\Omega)}$ in $E_2$.
 Therefore, in each example, we first investigate the impact of the bandwidth of $B_{S}$, and then take an appropriate width of $B_{S}$ for subsequent computation.

\medskip

Besides the symbols $E_1, E_2$ in \eqref{eq:local_est2}, we introduce new symbols as follows:
\begin{itemize}
\item The local error is denoted by%The local error and its estimation in \eqref{eq:local_est2} are denoted by 
$$
      E_{L} := \| \nabla u - \nabla u_h  \|_{S}.
      \quad
      \widehat{E}_{L} :=  \sqrt{E_1^2+E_2^2 } + 2C_0h \|f - \pi_h f\|_{\Omega}\: .
      $$
\item
       The global error and its estimation in \eqref{eq:global-error-estimation} are denoted by
$$ 
  E_{G} := \| \nabla u - \nabla u_h  \|_{\Omega},
  \quad \widehat{E}_{G} := \| \nabla u_h - \mathbf{p}_h  \|_{\Omega} + C_0h \|f -\pi_h f \|_{\Omega}.
  $$
\end{itemize}
Here, $u_h \in V_h$ and $\mathbf{p}_h \in RT_h$ are finite element solutions of the objective problems; $\mathbf{p}_h$ also satisfies the condition \eqref{eq:cond-ph}.  

\subsection{Square domain}
\label{subsec:square}
The error estimation proposed in this paper is applicable to problems with different boundary conditions. To illustrate this feature, let us consider the following Poisson equations over the unit square domain $\Omega=(0,1)^2$, where the subdomain is selected as $S=(0.375,0.625)^2$.
\begin{enumerate}
\item [(a)] Dirichlet boundary condition (exact solution $u = \sin(\pi x)\sin(\pi y)$).
\begin{equation}
        \label{eq:experiment}
          - \Delta u =  2 \pi^2 \sin(\pi x) \sin(\pi y) \mbox{ in } \Omega, \quad u = 0 \mbox{ on } \partial \Omega .
\end{equation}
\item [(b)] Neumann boundary condition (exact solution $u = \cos(\pi x)\cos(\pi y)$).
\begin{equation}
      \label{eq:experiment2}
            - \Delta u = 2 \pi^2 \cos(\pi x) \cos(\pi y) \mbox{ in } \Omega ,\quad  \frac{\partial u}{\partial \mathbf{n}}= 0 \mbox{ on } \partial \Omega,
             \quad  \int_\Omega u d x =0 \:.
\end{equation}
\end{enumerate}
 The finite element solutions $u_h, \mathbf{p}_h$ are computed with uniform meshes, and the mesh size $h$ here is chosen as the leg length of the triangle element for an uniform mesh.
%
%\medskip
    \begin{figure}[h!]
    \begin{center}
    \includegraphics[width=5.0cm]{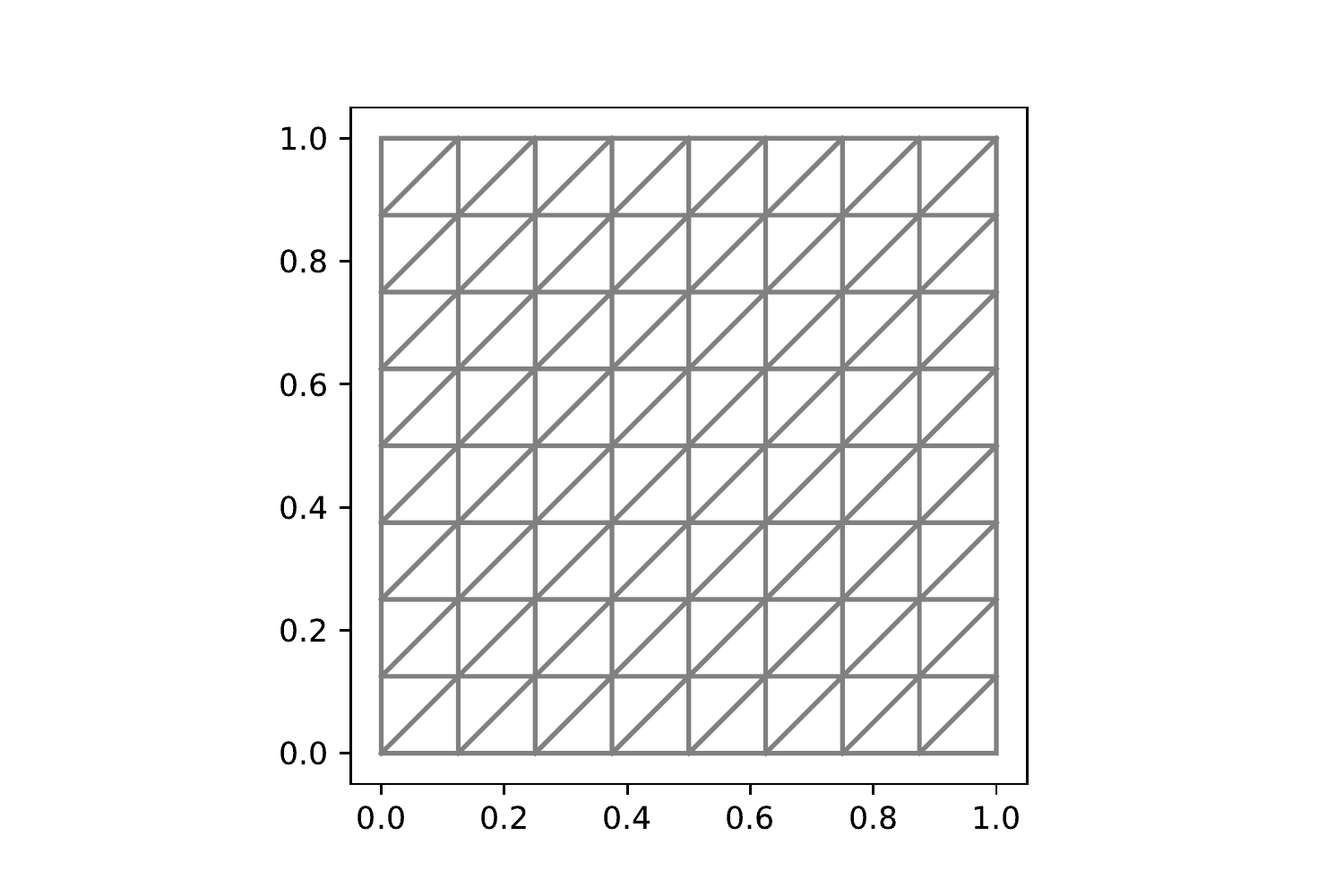}
    \includegraphics[width=5.0cm]{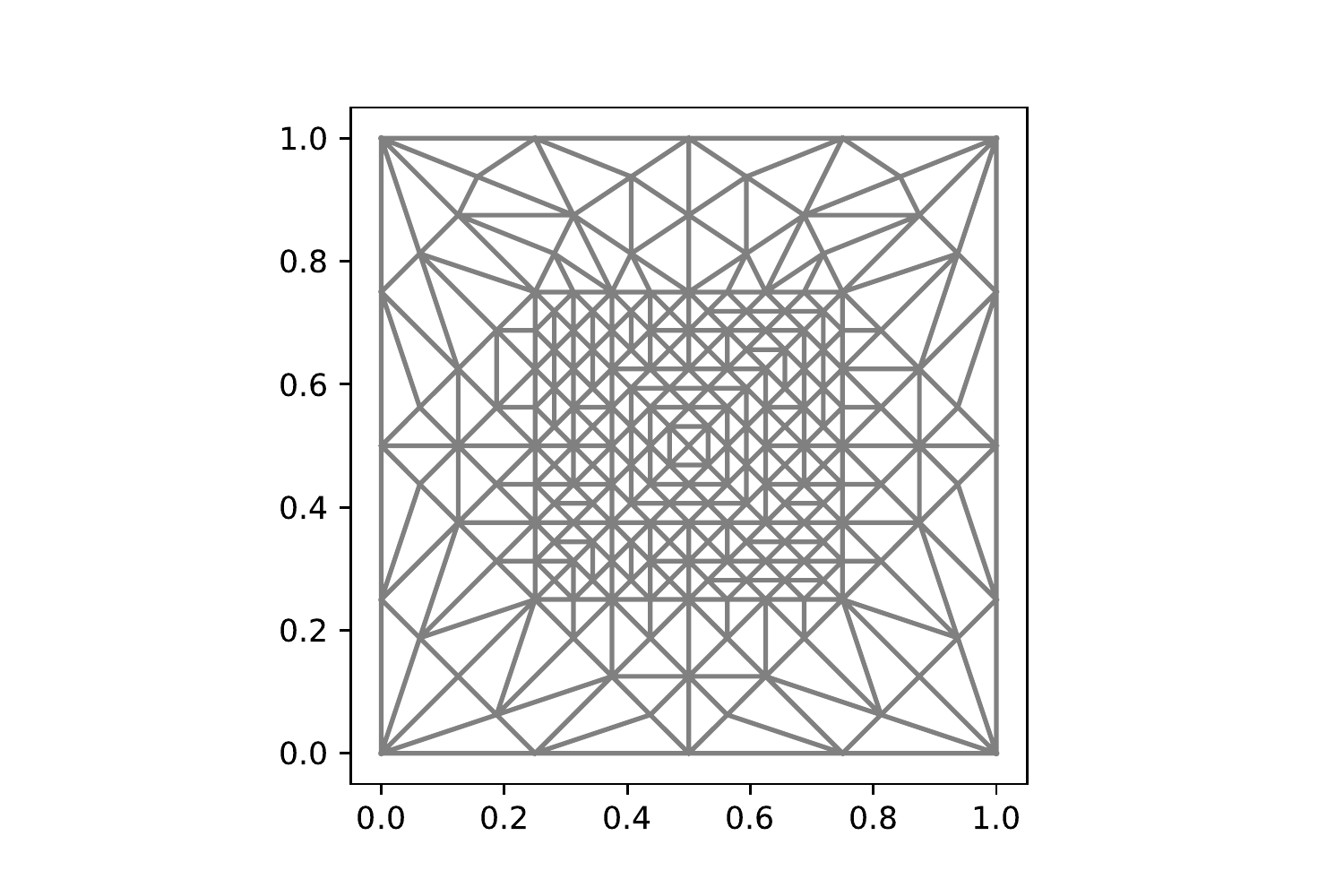}
    \caption{Uniform and non-uniform mesh (Rectangle domain). \label{fig:local_refined_mesh}}
    \end{center}
    \end{figure}
\begin{figure}[h]
	\begin{center}
	\includegraphics[width=3.5cm,angle=0]{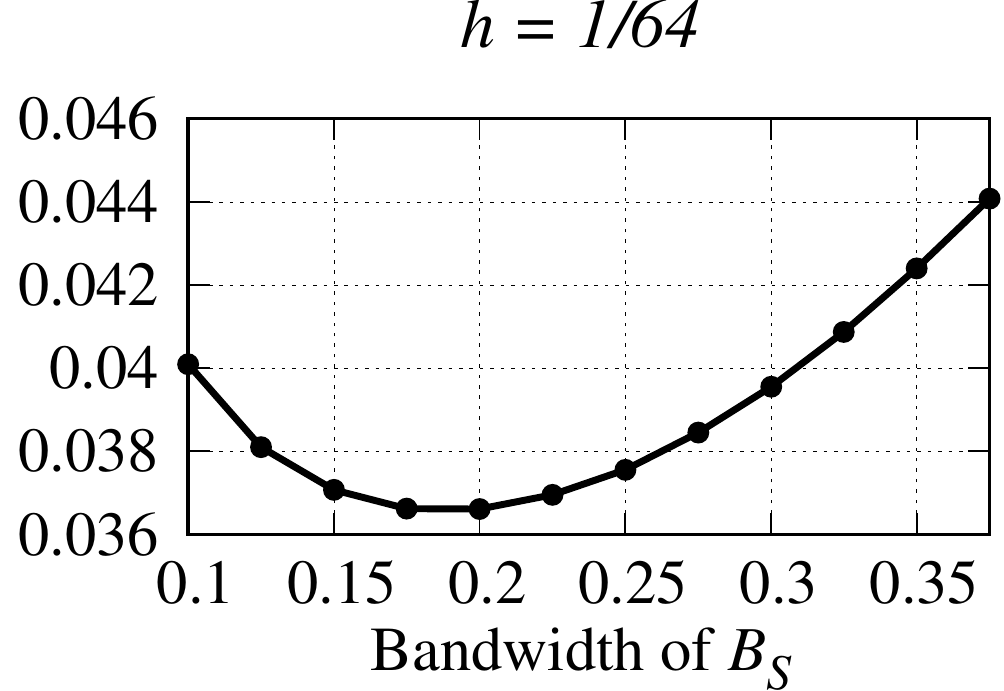} \quad
\includegraphics[width=3.5cm,angle=0]{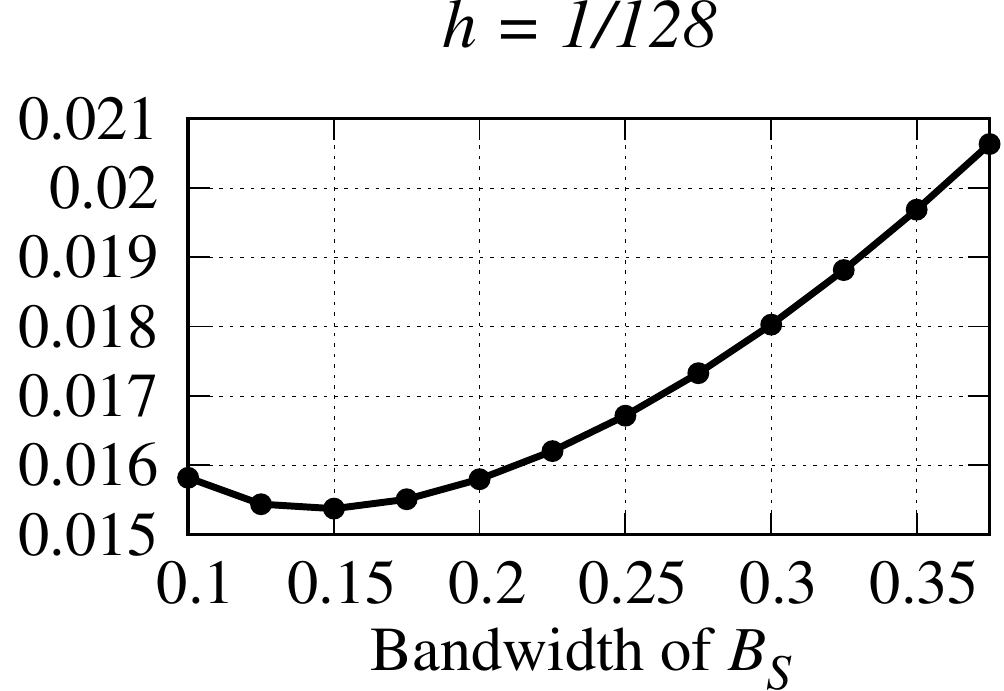}  \\
\includegraphics[width=3.5cm,angle=0]{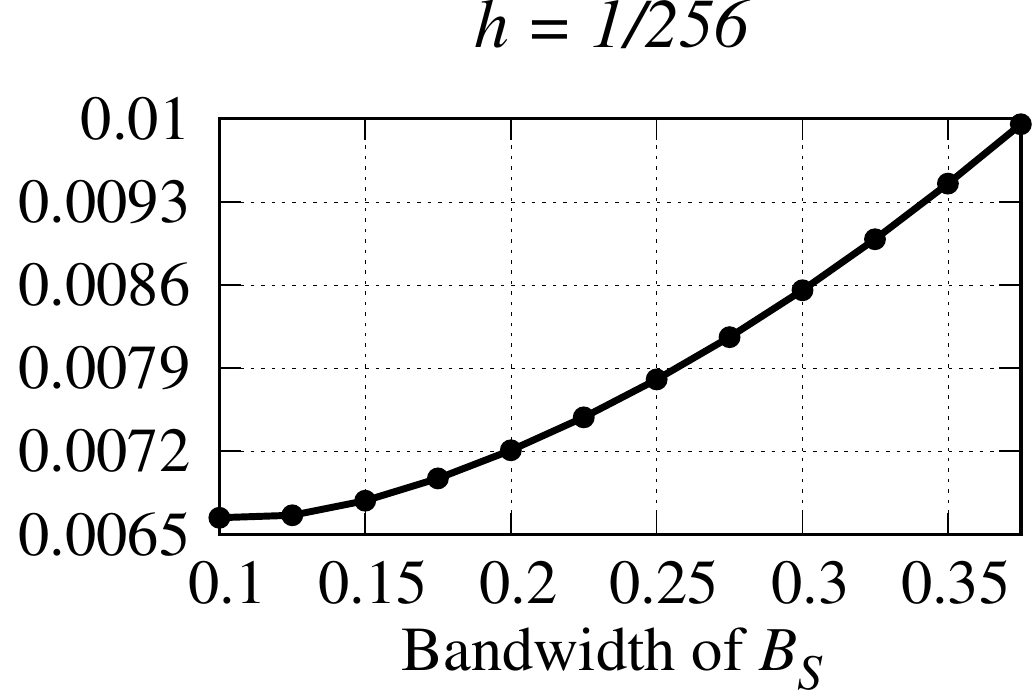} \quad
\includegraphics[width=3.5cm,angle=0]{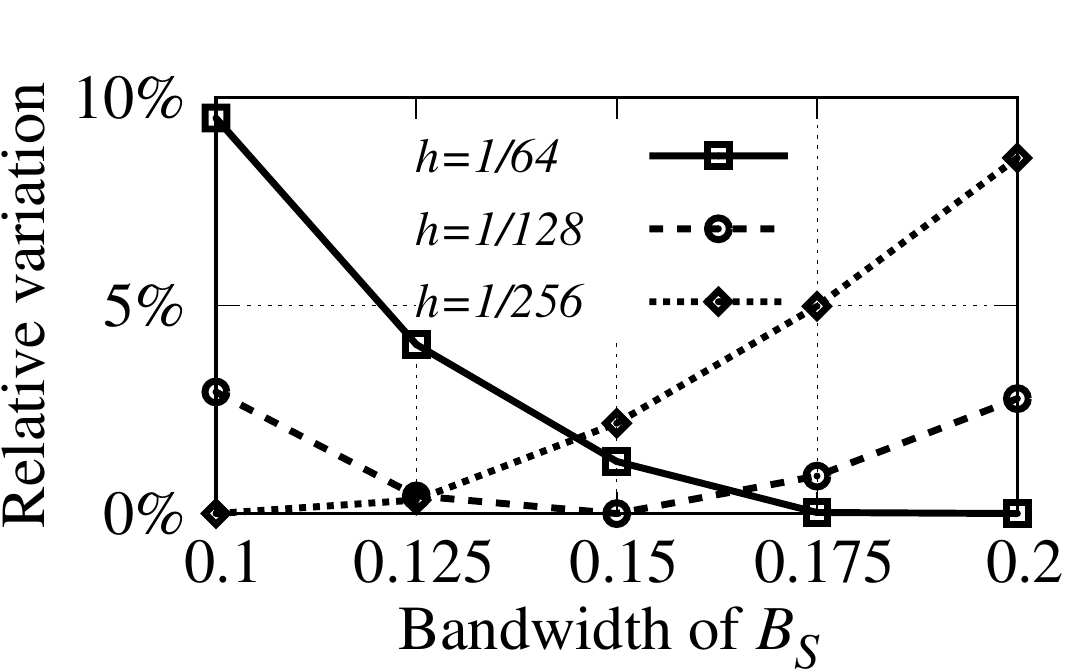}  \\\vspace{18pt}
%\hspace{50pt}  {(c) $h=1/256$} \hspace{80pt} {(d) Relative estimation variation} 
	\caption{Dependency of local error estimation on the bandwidth of $B_{S}$ (Dirichlet BVP and square domain, uniform mesh).}
	\label{fig:RD-tol}

	\includegraphics[width=3.5cm,angle=0]{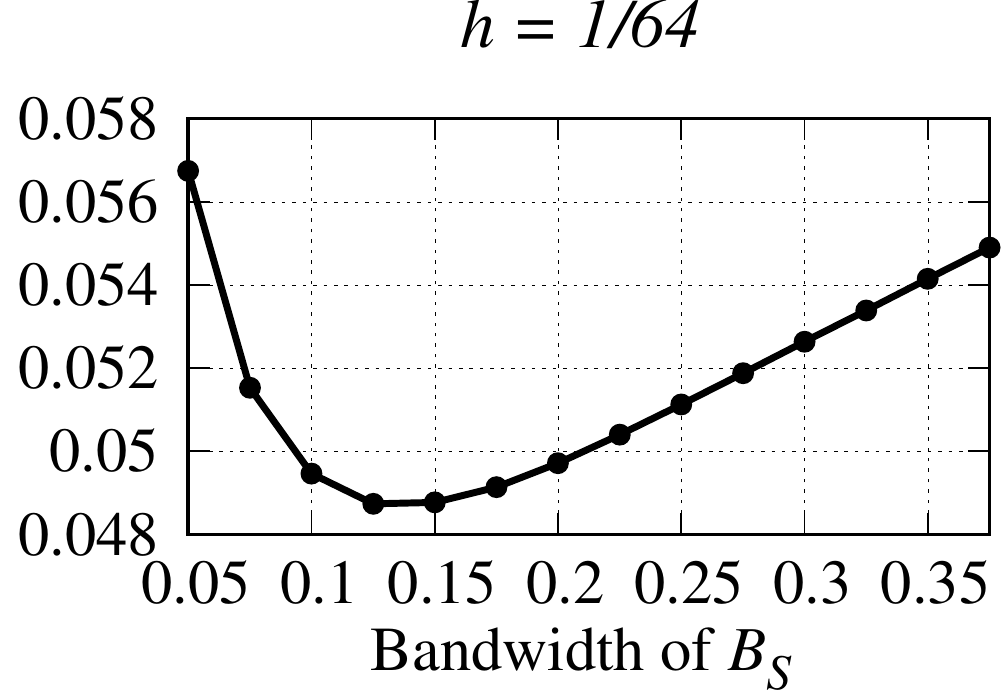} \quad
\includegraphics[width=3.5cm,angle=0]{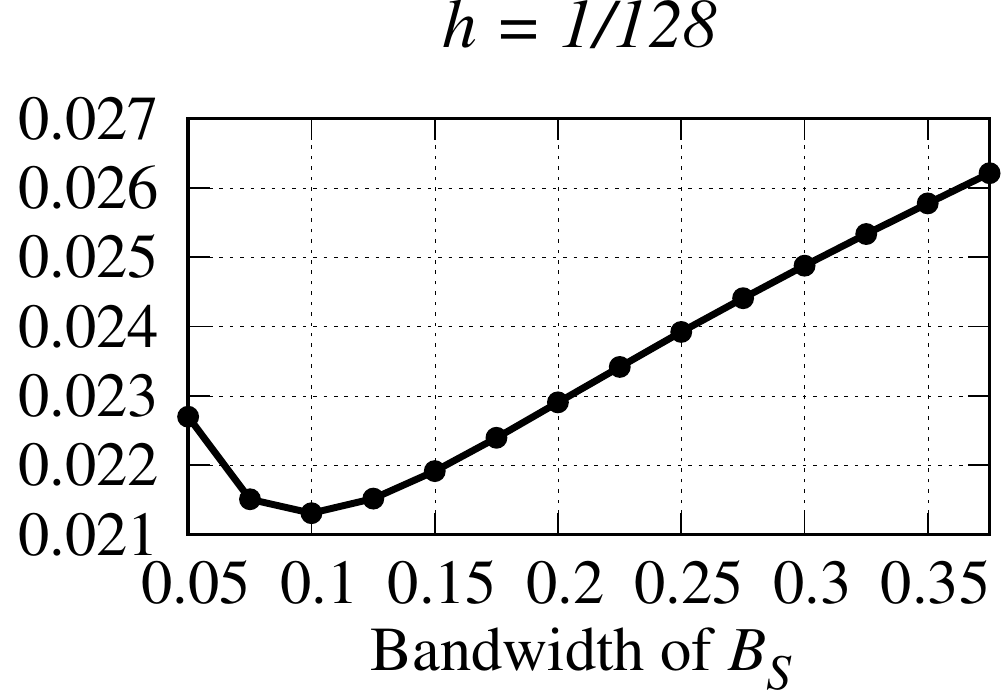}  \\
\includegraphics[width=3.5cm,angle=0]{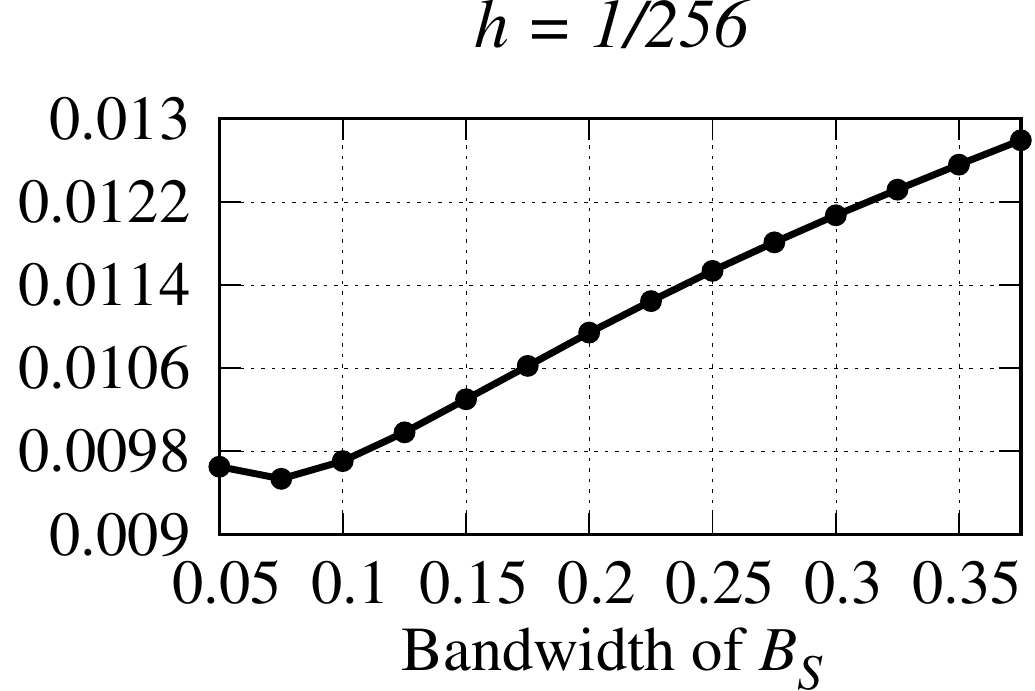} \quad
\includegraphics[width=3.5cm,angle=0]{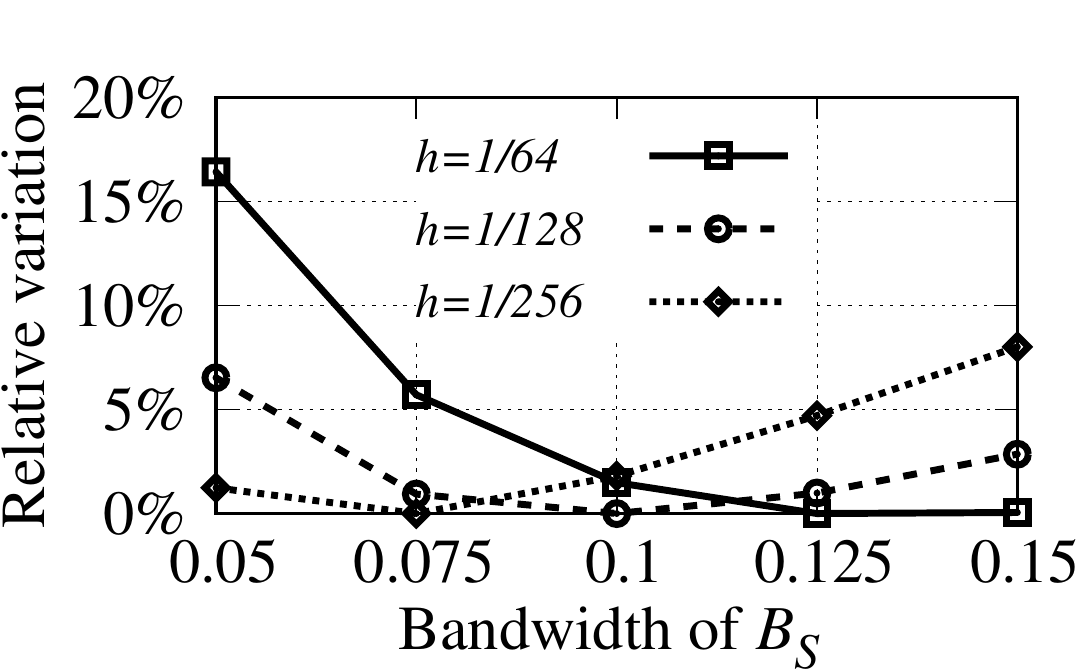}  \\\vspace{18pt}
	\caption{Dependency of local error estimation on the bandwidth of $B_{S}$ (Neumann BVP and square domain, uniform mesh).}
	\label{fig:RD-Neumann-tol}
	\end{center}
\end{figure}

{\medskip \bf \noindent Asymptotic behavior of the proposed local error estimator over a uniform mesh.} For Dirichlet and Neumann boundary conditions, the dependencies of the local error estimator $\widehat{E}_L$ on the bandwidth of $B_S$ are shown in Figure \ref{fig:RD-tol} and Figure \ref{fig:RD-Neumann-tol}, respectively.
The relative variation of local error estimator with respect to bandwidth selection is displayed for two problems. It is noteworthy that the local error estimation is not significantly sensitive to variations in bandwidth. 
  For example, in Figure \ref{fig:RD-tol}, for $h=1/64$, the relative variation in error estimation with respect to a bandwidth in the range $[0.125,0.275]$ is less than 5\%.  

 In the following discussion, the bandwidth of $B_{S}$ is selected as $0.15$ for the Dirichlet boundary condition and $0.10$ for the Neumann boundary condition.

\begin{figure}[thp]
	\begin{center}
	\includegraphics[angle=0,width=5.6cm]{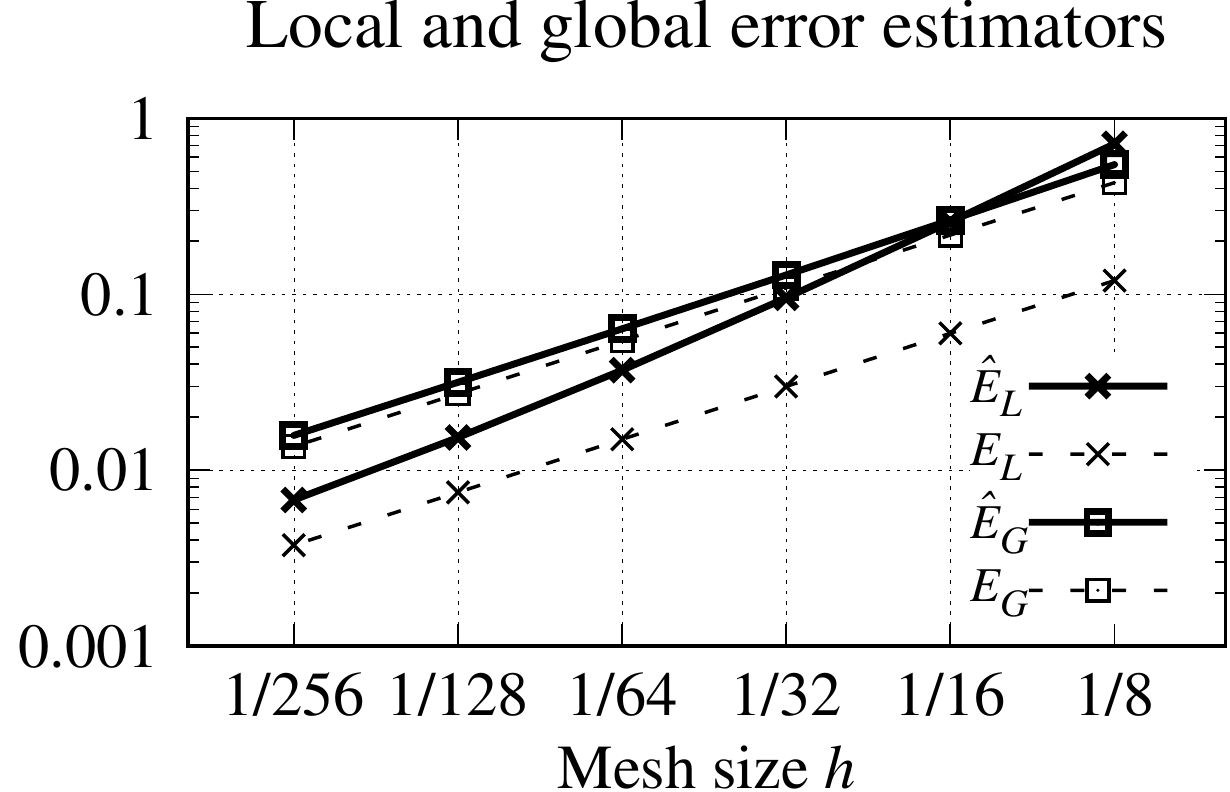} \hspace{5pt}
	\includegraphics[angle=0,width=5.6cm]{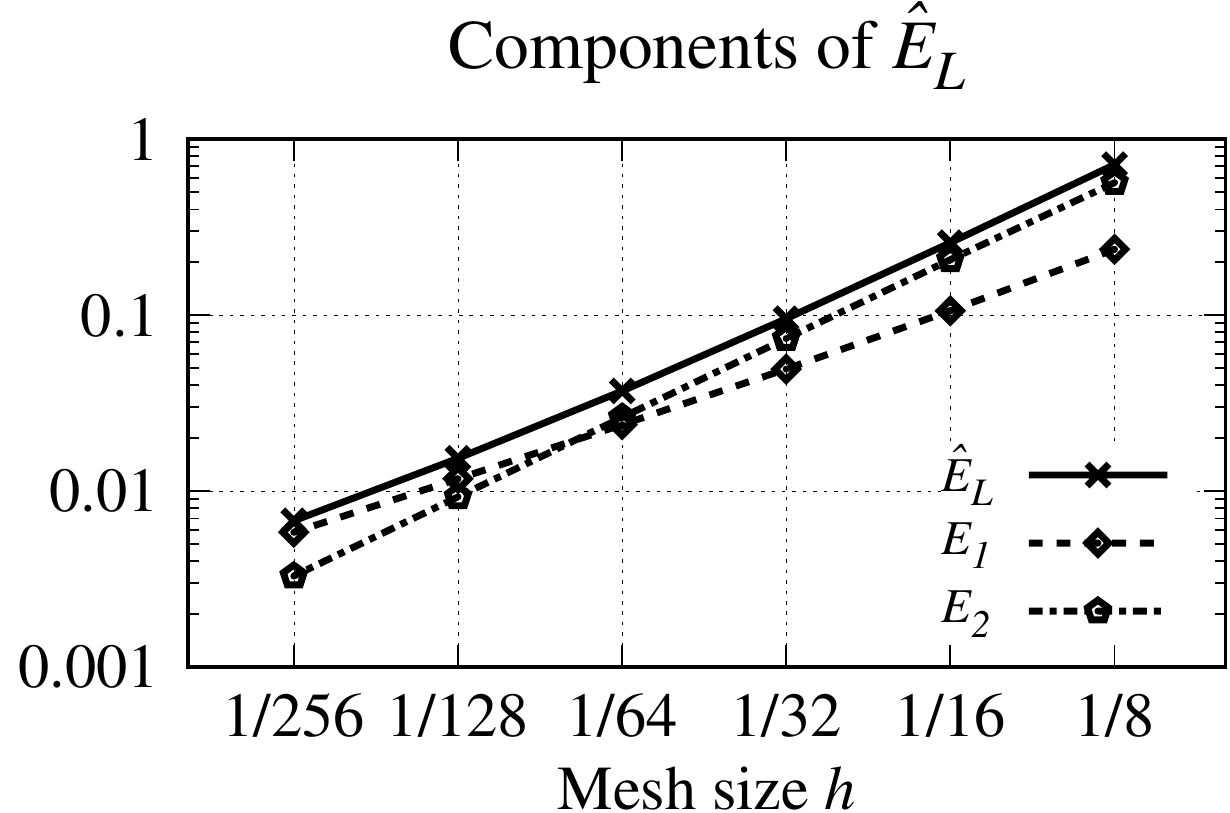} \\\vspace{18pt}
	\caption{Error estimators for Dirichlet BVP (square domain, uniform mesh). }
	\label{fig:Errest_S}
	\end{center}

	\begin{center}
	\includegraphics[angle=0,width=5.6cm]{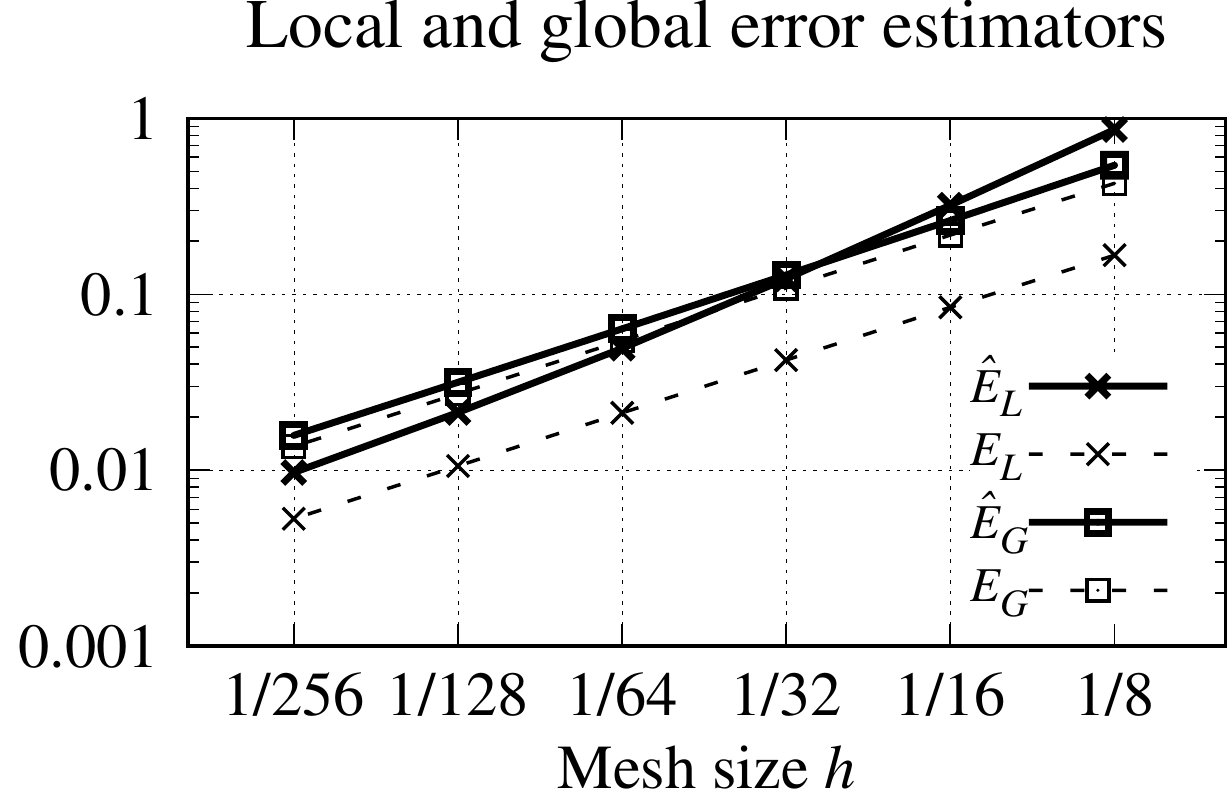} \hspace{5pt}
	\includegraphics[angle=0,width=5.6cm]{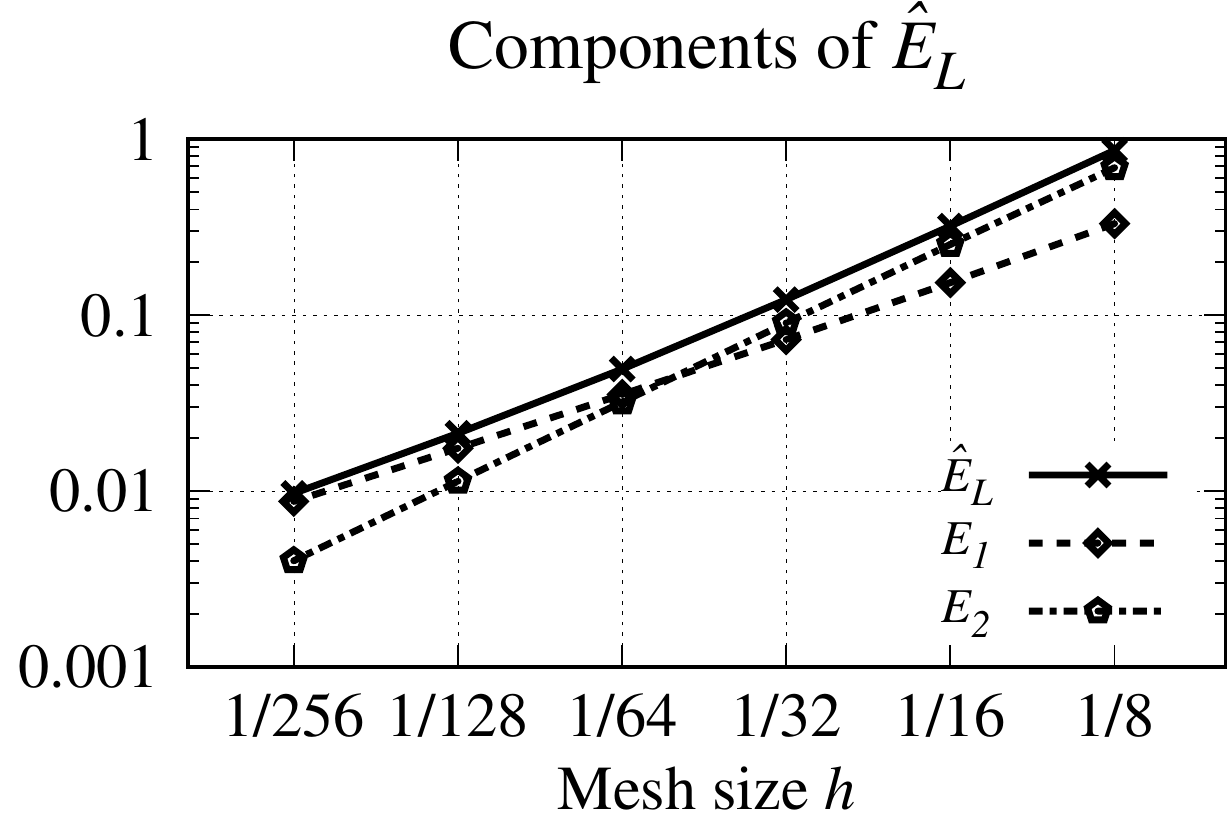} \\\vspace{15pt}
	\caption{Error estimators for Neumann BVP (square domain, uniform mesh). }
	\label{fig:Errest_Neumann_S}
	\end{center}
\end{figure}
\begin{table}[thp]
\caption{Error estimate for Dirichlet BVP (square domain , uniform mesh)\label{table:result_S}}

% \scalebox{0.65}[0.65]{ 
\begin{center}
\begin{tabular}{|c|cc|c|cc|cc|}
\hline
\rule[-0.05cm]{0pt}{0.4cm}{} 
$h$&$\kappa_h$&$C(h)$&$E_{L}$&$E_1$&$E_2$&$\widehat{E}_{L}$&$\widehat{E}_{G}$\\\hline\hline
      1/16  &  0.030 &0.036  & 0.060  & 0.106 & 0.206 & 0.258  & 0.264      \\
      1/32  &  0.015 & 0.018  & 0.030  & 0.049 & 0.074 & 0.095  & 0.129      \\
      1/64  & 0.008 & 0.009  & 0.015  & 0.024 & 0.026 & 0.037  & 0.064      \\
      1/128 & 0.004 &0.005  & 0.007   & 0.012 & 0.009 & 0.015  & 0.032      \\
      1/256 & 0.002 &0.002  & 0.004   & 0.006 & 0.003 & 0.007  & 0.016      \\
\hline
\end{tabular}
\end{center}
\end{table}

\begin{table}[thp]
\caption{Error estimate for Neumann BVP (square domain, uniform mesh).}
\begin{center}
\begin{tabular}{|c|cc|c|cc|cc|}
\hline
\rule[-0.05cm]{0pt}{0.4cm}{} 
$h$&$\kappa_h$&$C(h)$&$E_{L}$&$E_1$&$E_2$&$\widehat{E}_{L}$&$\widehat{E}_{G}$\\\hline\hline
      1/16  &  0.030 &0.036  & 0.084   & 0.153  & 0.252 & 0.320  & 0.263      \\
      1/32  &  0.015 & 0.018  & 0.042    & 0.073 & 0.090  & 0.122  & 0.129      \\
      1/64  & 0.008 & 0.009  & 0.021   & 0.036 & 0.032  & 0.049  & 0.064      \\
      1/128 & 0.004 &0.005  & 0.011   & 0.018 & 0.011  & 0.021  & 0.032      \\
      1/256 & 0.002 &0.002  & 0.005    & 0.010 & 0.004  & 0.010  & 0.016      \\
\hline
\end{tabular}
\label{table:result_Neumann_S}
\end{center}
\end{table}
   A detailed discussion on each component of the error estimators is also presented; see Table \ref{table:result_S} and Figure \ref{fig:Errest_S} for Dirichlet boundary condition, Table \ref{table:result_Neumann_S} and Figure \ref{fig:Errest_Neumann_S} for Neumann boundary condition. From the numerical results, we confirm that for both the problems the main term $E_1$ of the error estimation \eqref{eq:local_est2} becomes dominant when $h \leq 1/128$, which agrees with the analysis in \S \ref{sec:5}. 
  
{
 \noindent \textbf{Improved convergence rate of the global error term  $E_2$.} Here, we confirm the improvement of the local error estimator discussed in \S \ref{sec:5}, when the approximation $\mathbf{p}_h$ to $\nabla u$ is obtained in a higher degree Raviart-Thomas FEM space. 
 Table \ref{table:improve_global_order} and Figure \ref{fig:mesh_dependency_improved} show that the global error term $E_2$ has an improved convergence rate as $O(h^2)$ by using $\mathbf{p}_h \in RT_h^1$. The numerical results support the theoretical result \eqref{eq:improved_global_order}.}
    \begin{figure}[thp]
    	\begin{center}
    	\includegraphics[angle=0,width=5.6cm]{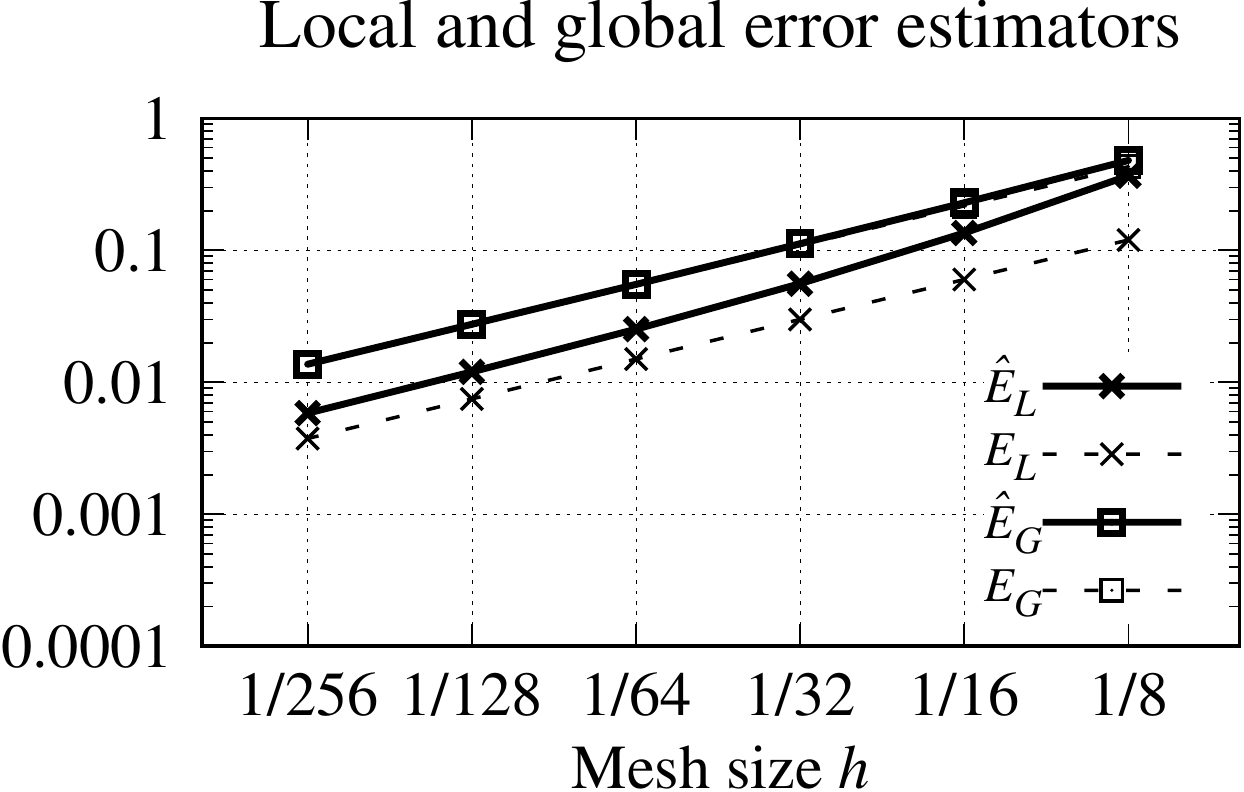}
    	\includegraphics[angle=0,width=5.6cm]{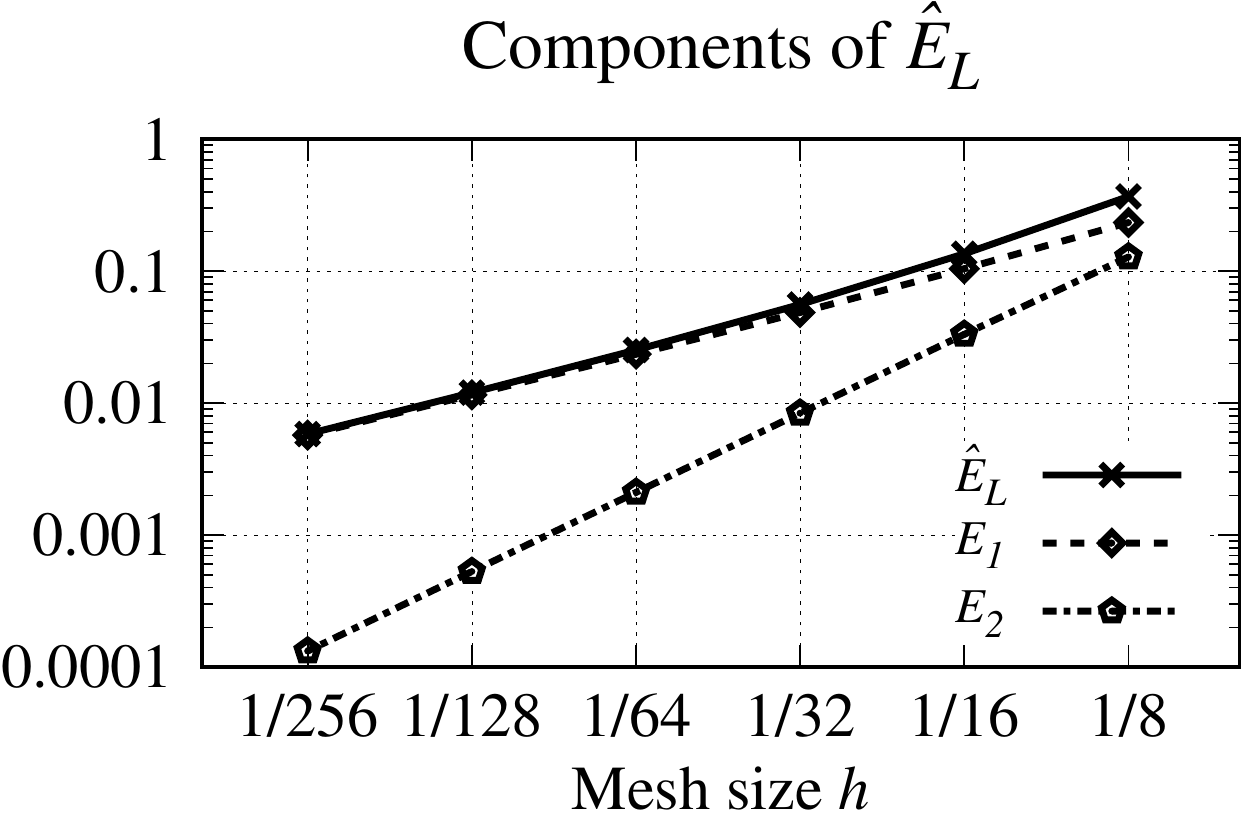} 
    	 \\\vspace{18pt}
    	\caption{ Improved convergence rate of the local error estimator \eqref{eq:improved_global_order} with $\overline{w}_h \in V_h^{2}$, $\mathbf{p}_h \in RT_h^{1}$ (uniform mesh for squared domain). }
    	\label{fig:mesh_dependency_improved}
    	\end{center}
    \end{figure}
    \begin{table}[thp]
    \caption{Convergence rate of $E_2$ w.r.t. the selection of $\mathbf{p_h}  \in RT_{h}^{k-1}$.}
    \begin{center}
    \begin{tabular}{|c|cc|cc|}
    \hline
    \rule[-0.05cm]{0pt}{0.4cm}{} 
    $h$ & $E_2$ $(k=1)$   & Order & $E_{2} (k=2)$ & Order  \\ \hline\hline 
          1/16    & 0.206   & -     & 0.036 & -     \\
          1/32    & 0.074   & 1.484 & 0.009 & 1.982      \\
          1/64    & 0.026   & 1.493 & 0.002 & 1.992   \\
          1/128   & 0.009   & 1.497 & 5.7e-4  & 1.997     \\
          1/256   & 0.003   & 1.499 & 1.4e-4  & 1.999       \\
    \hline
    \end{tabular}
    \label{table:improve_global_order}
    \end{center}
\end{table}

 %\newpage
 %\medskip
 {
 \noindent \textbf{Convergence behavior for  non-uniform meshes.}
Based on numerical results, we investigate the behavior of our proposed estimator \eqref{eq:local_est2} for non-uniform meshes under the setting $h_G = O(\sqrt{h_{\Omega'}})$; see a sample non-uniform mesh in Figure \ref{fig:local_refined_mesh}. The subdomain $S$ and the bandwidth $B_S$ are set to $(0.375,0.625)^2$ and $0.125$, respectively. 
 
\begin{figure}[h!]
	\begin{center}
	\includegraphics[angle=0,width=5.6cm]{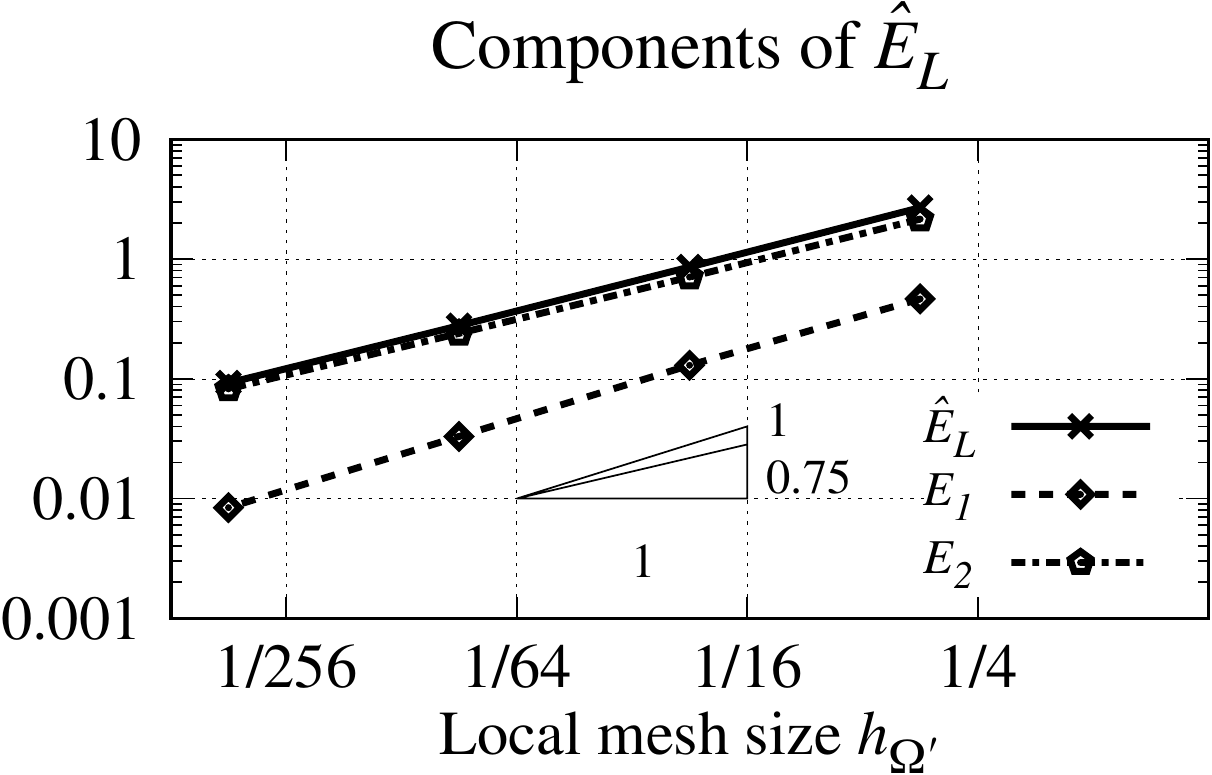} %\\\vspace{18pt}
	\includegraphics[angle=0,width=5.6cm]{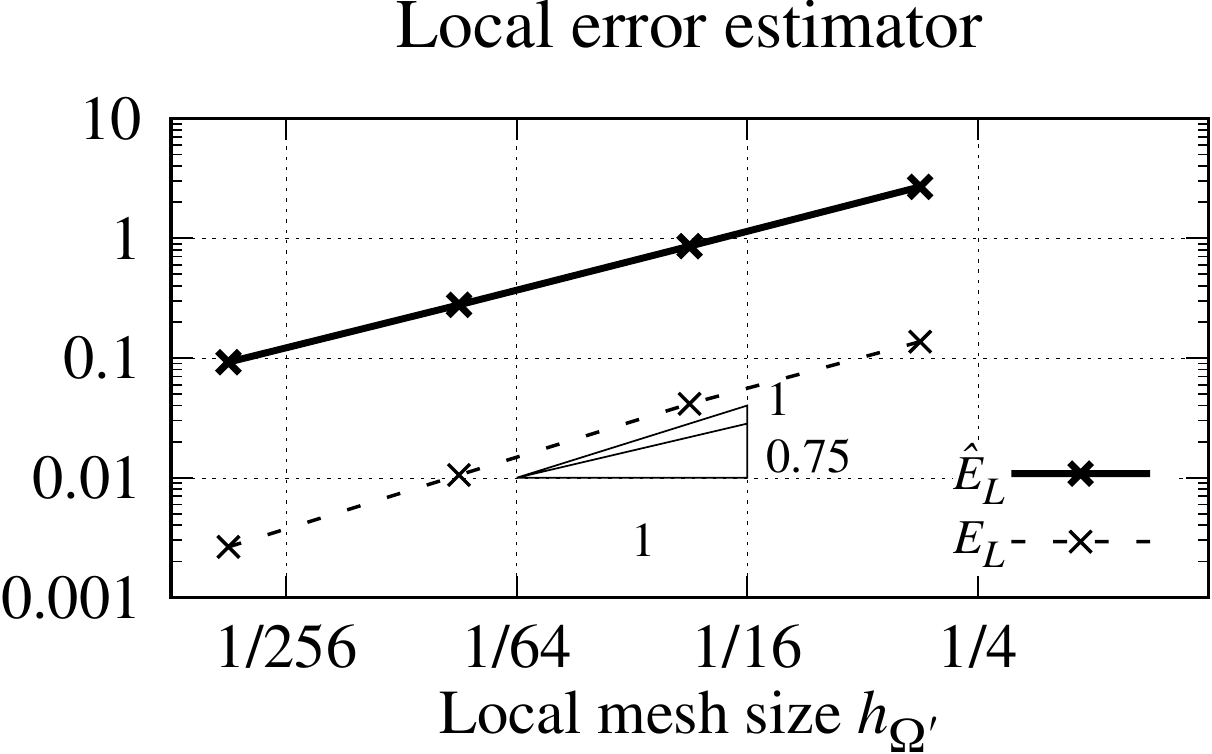} \\\vspace{18pt}
	\caption{ Local error estimator for lowest order FEM over non-uniform mesh (squared domain). }
	\label{fig:lowest_local_mesh_dependency}
	\end{center}
\end{figure}
  
    \begin{figure}[h!]
    	\begin{center}
    	\includegraphics[angle=0,width=5.6cm]{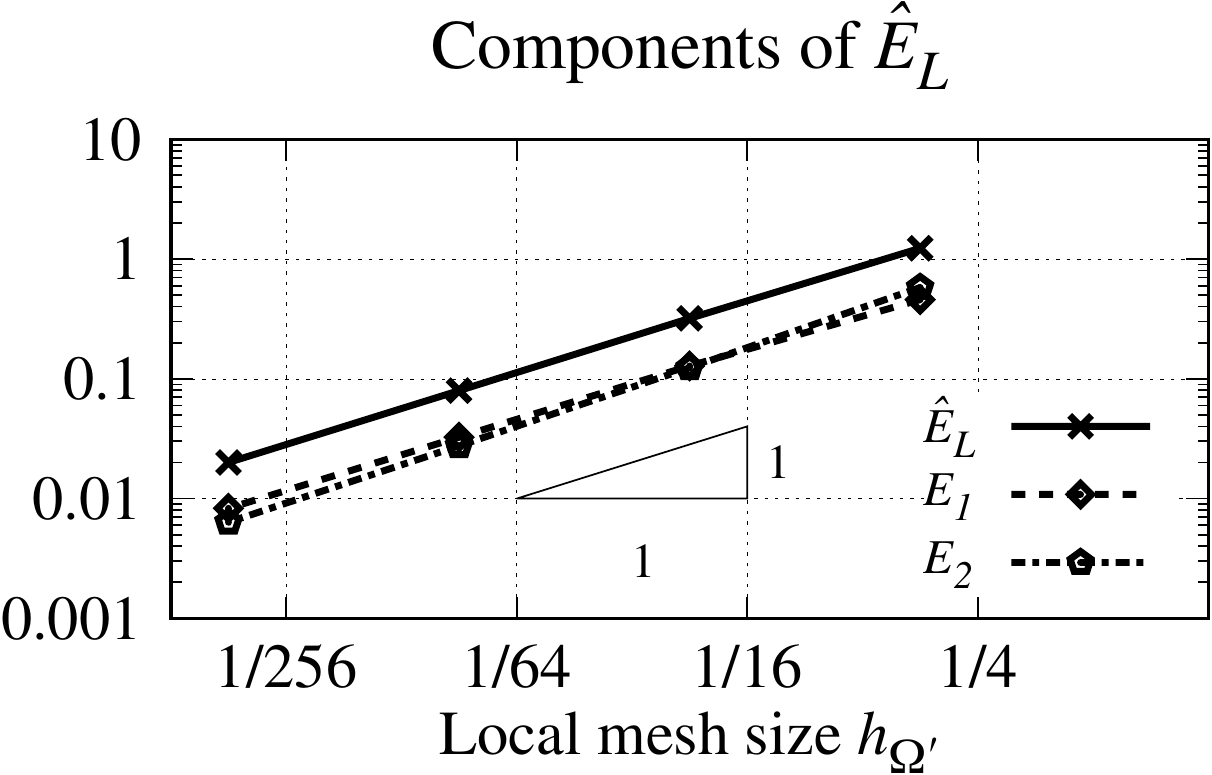} 
    	\includegraphics[angle=0,width=5.6cm]{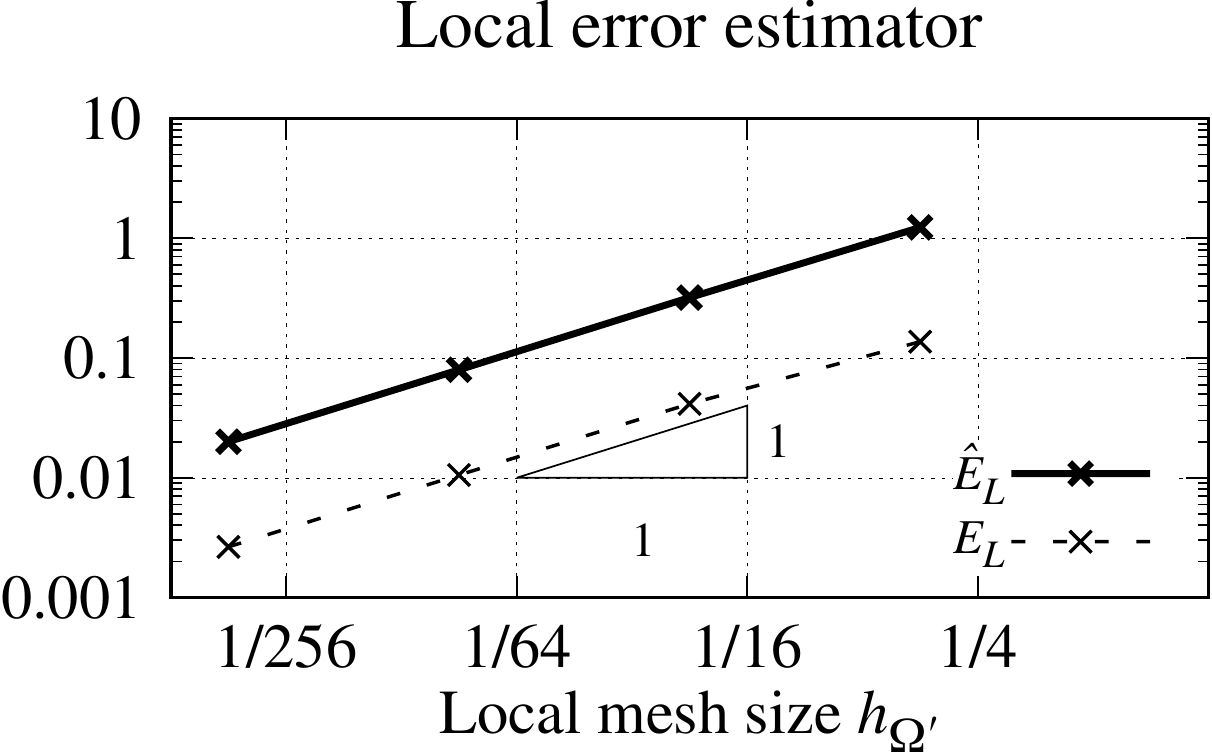} \\\vspace{18pt}
    	\caption{ Improvement of the local error estimator \eqref{eq:local_est2} for FEM of higher degree over non-uniform meshes (squared domain). }
    	\label{fig:local_mesh_dependency_imporoved}
    	\end{center}
    \end{figure}
    
    \begin{table}[h!]
    \caption{Convergence rate of $\|\nabla u_h - \mathbf{p}_h \|_{\alpha}$ over non-uniform mesh (squared domain).}
    \begin{center}
    \begin{tabular}{|c|c|c|}
    \hline
    \rule[-0.05cm]{0pt}{0.4cm}{} 
    $h_{\Omega'}$ & $\|\nabla u_h - \mathbf{p}_h \|_{\alpha} $  & Order   \\ \hline\hline 
          0.177    & 0.221   & -       \\
          0.044   & 0.059   & 0.957     \\
          0.011    & 0.015   & 1.007   \\
          0.003   & 0.004   & 0.997      \\
    \hline
    \end{tabular}
    \label{table:uh_2_ph_order}
    \end{center}
\end{table}
    %\newpage
     The numerical results in Figure \ref{fig:lowest_local_mesh_dependency} and Table \ref{table:uh_2_ph_order} show that the convergence rates of  $\|\nabla u_h - \mathbf{p}_h \|_{\alpha}$ and $E_1$ are almost $O(h_{\Omega'})$.
     The numerical results support the expectation that $\|\nabla u_h - \mathbf{p}_h \|_{\alpha} = O(h_{\Omega'})$ under current mesh configuration (i.e., $h_G =\sqrt{h_{\Omega'}}$).
      In case that the lowest degree Raviart-Thomas FEM employed to compute the global term $E_2$, the convergence rate of the estimator $\widehat{E}_L$ is approximately $O(h_{\Omega'}^{0.75})$.
      When  $\tilde{p}_h$ is selected from $RT^{(1)}$,  
      the convergence rate of $E_2$ and $\widehat{E}_L$ is improved to be $O(h_{\Omega'})$;  
     see Figure \ref{fig:local_mesh_dependency_imporoved}.
     The theoretical convergence rate of $\|\nabla u_h - \mathbf{p}_h \|_{\alpha}$ will be considered in our succeeding research.
     
     As a conclusion, our proposed local error estimator is dominated by the local error term $E_1 = O(h_{\Omega'})$. Thus, it is possible to increase the efficiency of computation by using 
     a non-uniform mesh with a raw triangulation for the subdomain outside of the part of interest.}

%\newpage
\subsection{L-shaped domain}

\begin{figure}[h!]
\begin{center}
\includegraphics[width=4cm]{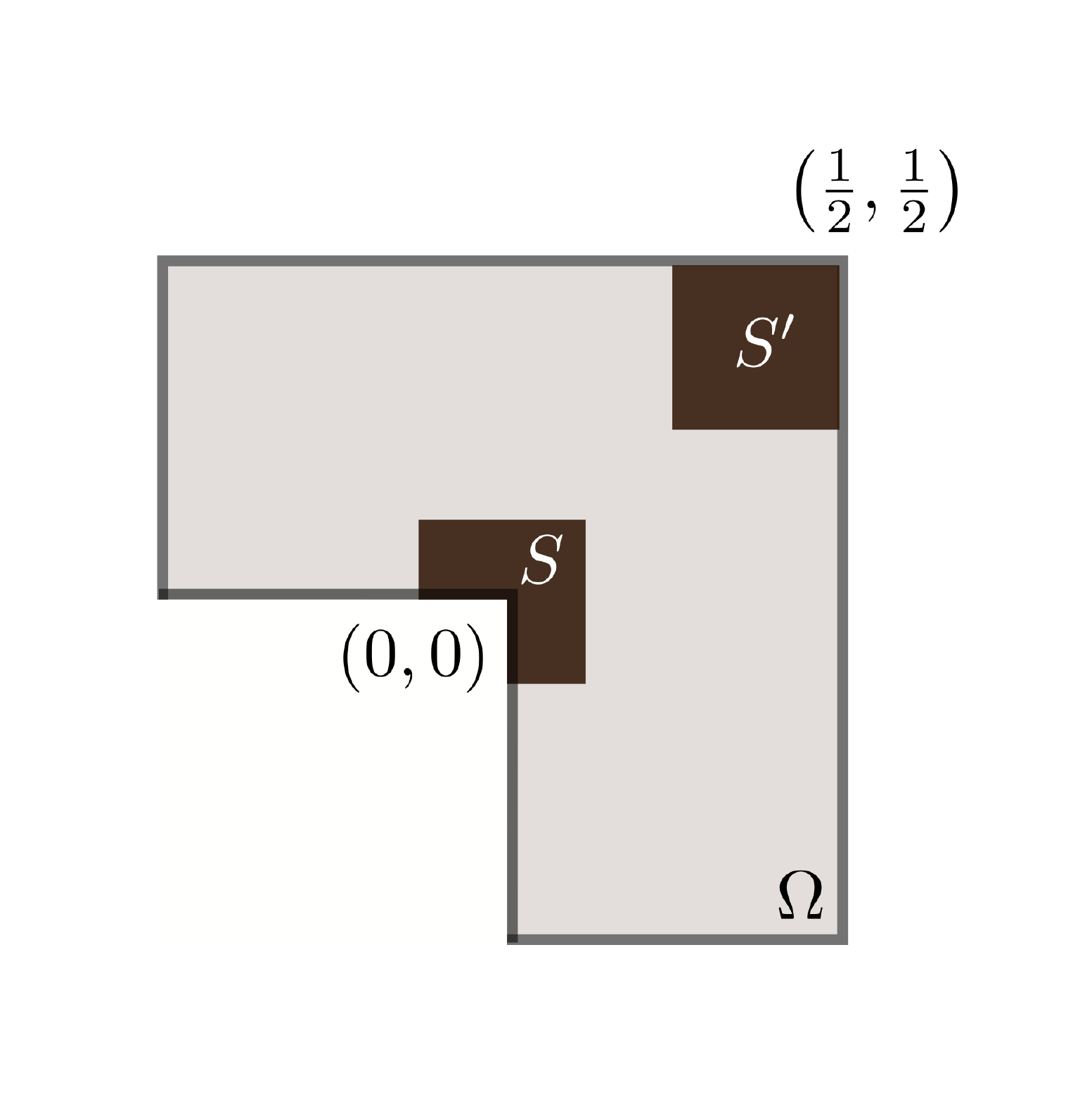} \hspace{-20pt}
\includegraphics[width=5.2cm]{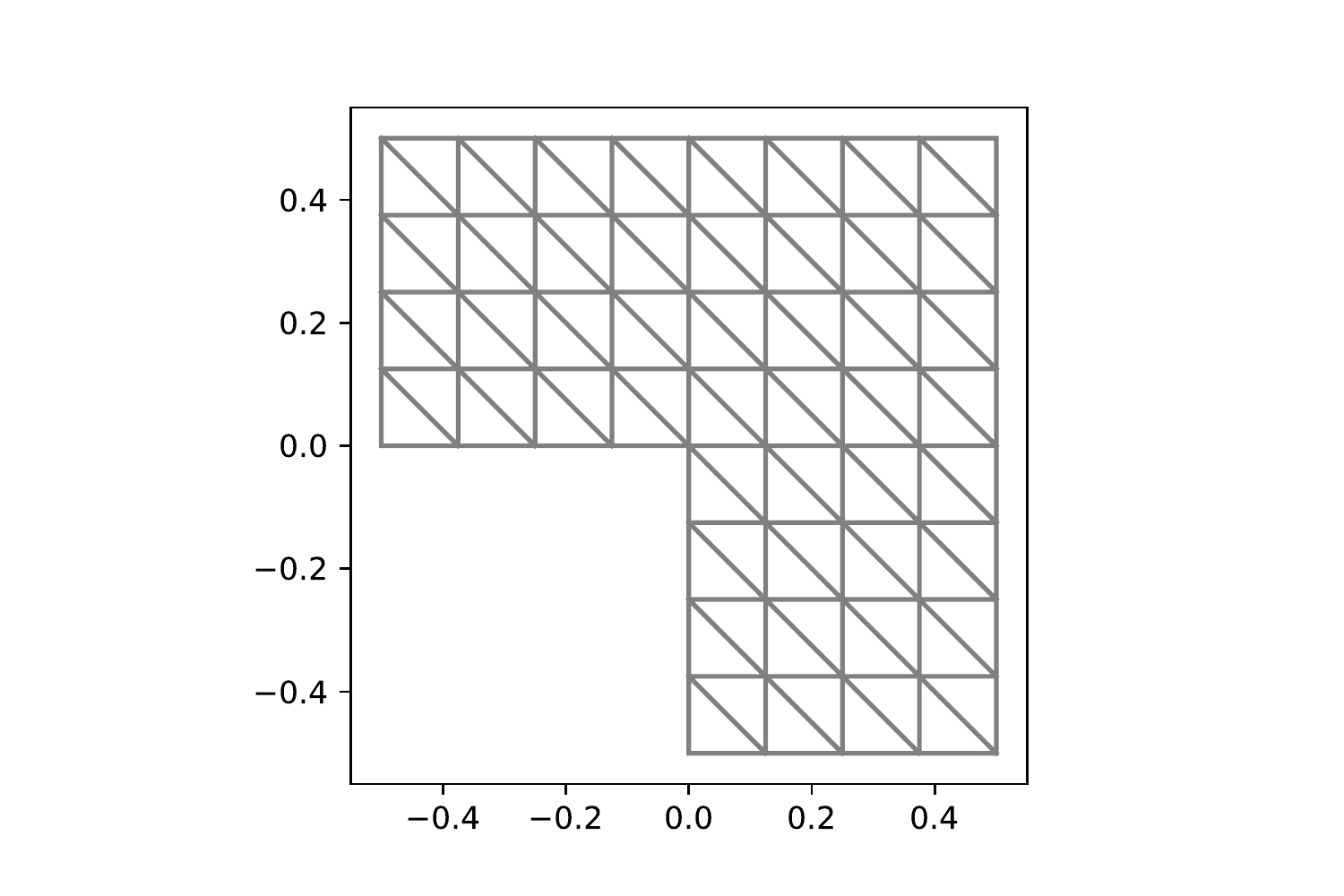} \hspace{-40pt}
\includegraphics[width=5.2cm]{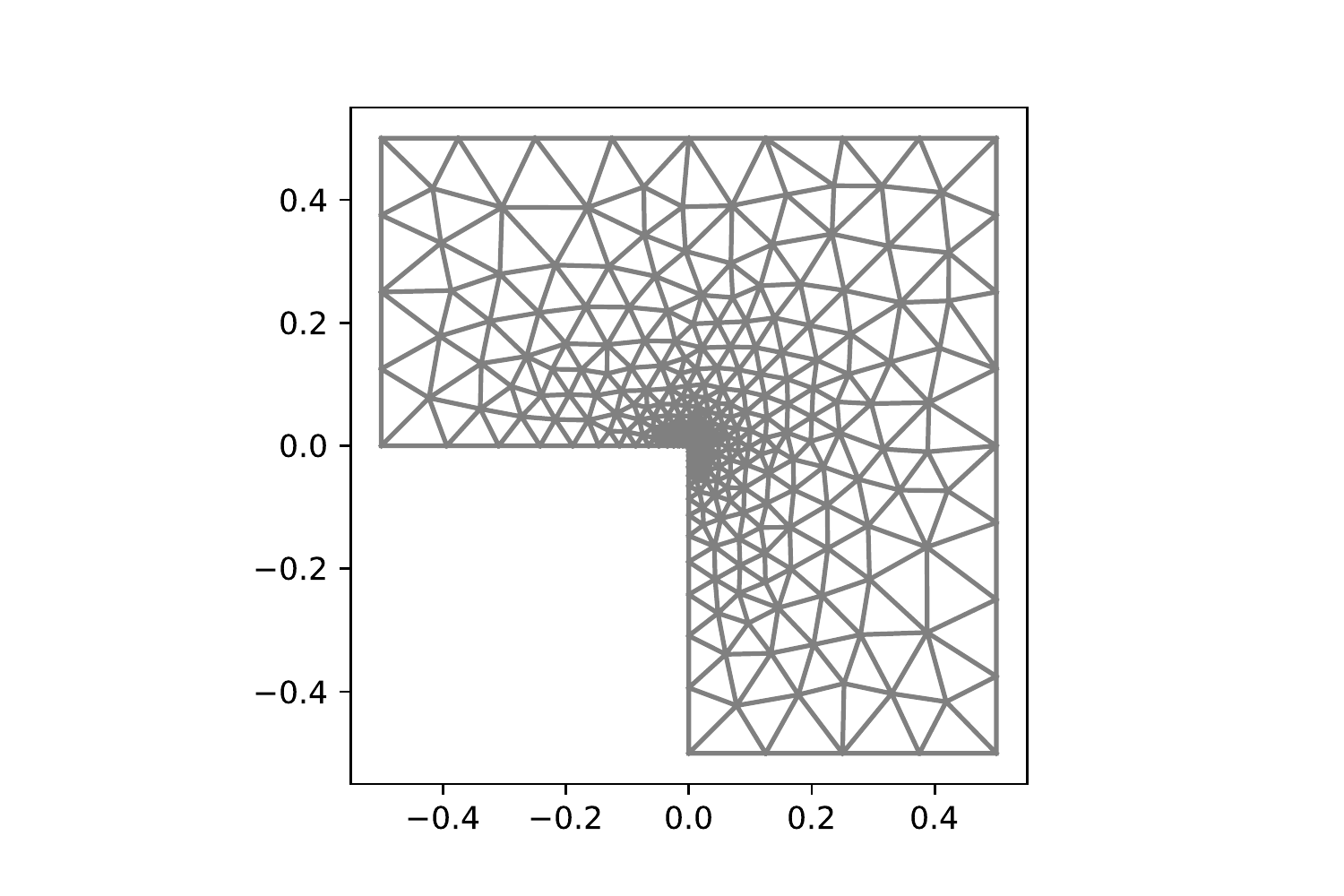}
\caption{Uniform and non-uniform mesh for an L-shaped domain $\Omega$ with two subdomains $S$, $S'$. \label{fig:L-shaped-domain}}
\end{center}
\end{figure}

 The proposed error estimation \eqref{eq:local_est2} is applicable to problems with a singular solution and the even case in which the subdomain $S$ and $\Omega$ share a common part of the boundary.
   In this sub-section, we consider the boundary value problem over an L-shaped domain $\Omega := (-0.5,0.5)^2 \setminus [-0.5,0]^2$; see Figure \ref{fig:L-shaped-domain}.
   The error estimation on two subdomains $S=\Omega \cap  (-0.125,0.125)^2$ and $S'=(0.25,~0.5)^2$ will be considered.

Let $u=r^{\frac{2}{3}} \sin{\left ( \frac{2}{3} (\theta + \frac{\pi}{2})\right )} \cos{(\pi x)}\cos{(\pi y)}$, where
$r$ and $\theta$ are the variables under the polar coordinates. Define $f=-\Delta u$. 
Then $u$ is the solution of the following equation.
\begin{eqnarray}
  - \Delta u =  f \mbox{ in } \Omega,\quad u = 0 \mbox{ on } \partial\Omega .
\label{eq:experiment3}
\end{eqnarray}
It is easy to confirm that $u \notin H^2(\Omega)$ due to the singularity around the re-entry corner point of the domain.

\noindent\textbf{Selection of the bandwidth of $B_{S}$.} The dependency of $\widehat{E}_{L}$ on the bandwidth of the $B_S$ is shown in Figure  \ref{fig:RD-tol-L}. 
 In the following computation, the bandwidth of $B_S, B_{S'}$ is selected as $0.375$, $0.25$, respectively.

\begin{figure}[h!]
\begin{center}
\includegraphics[width=3.5cm,angle=0]{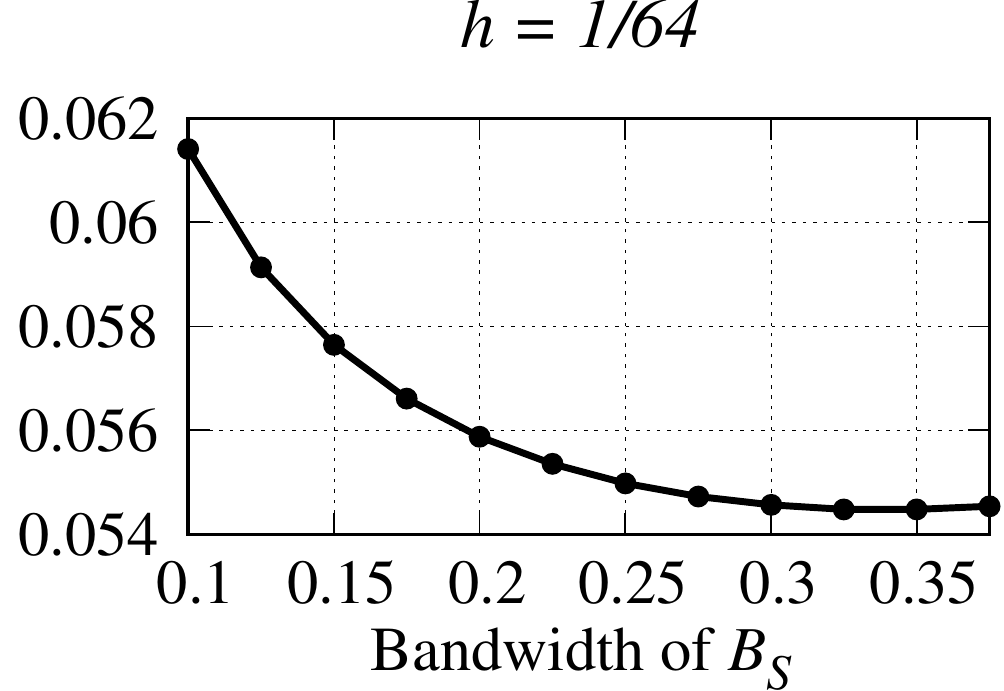} \quad
\includegraphics[width=3.5cm,angle=0]{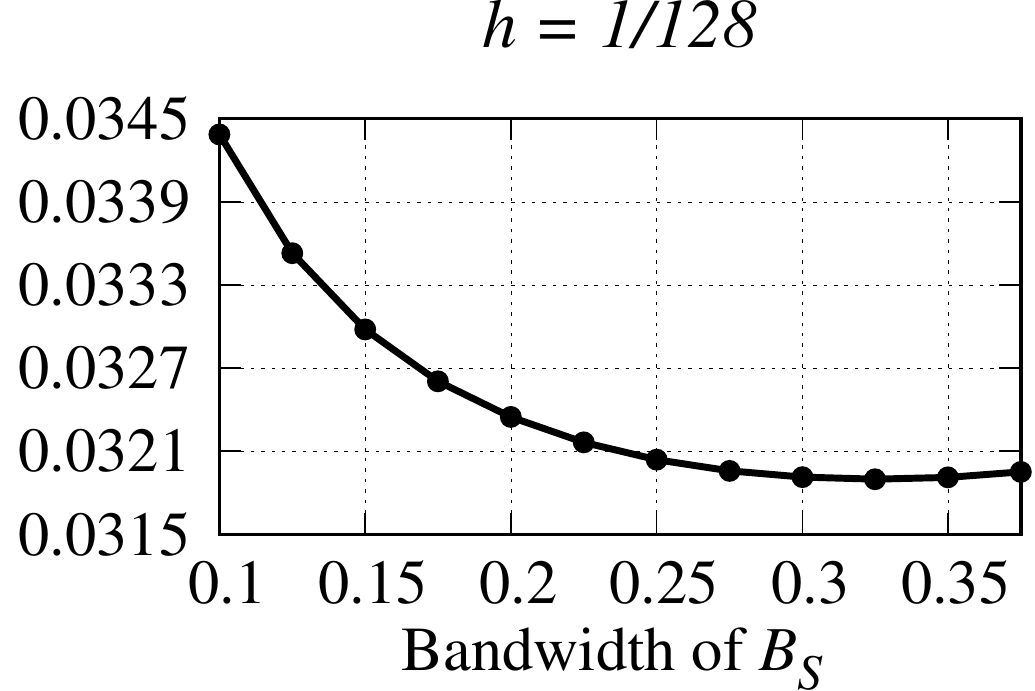}  \\
\includegraphics[width=3.5cm,angle=0]{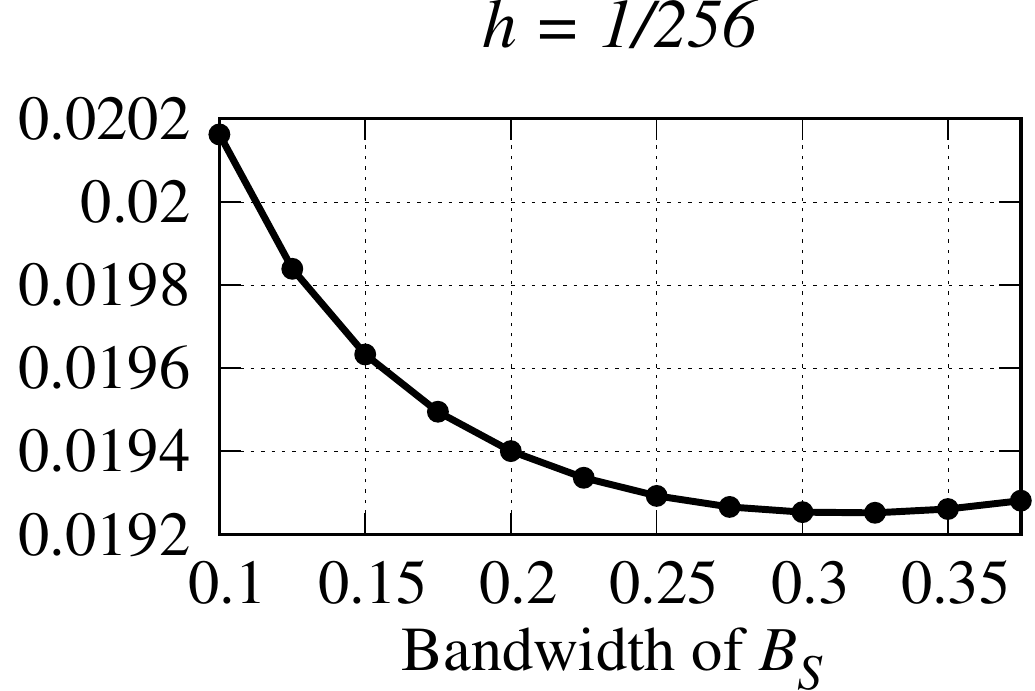}  \quad
\includegraphics[width=3.5cm,angle=0]{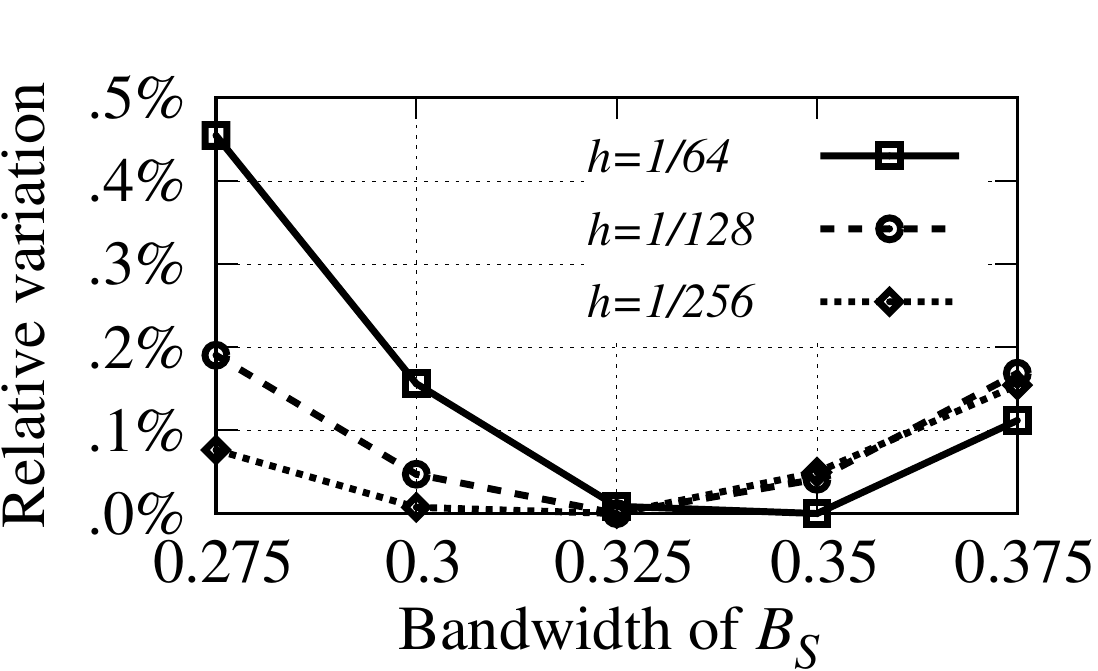}  \\
\vspace{10pt}
\caption{Dependency of local error estimation on the bandwidth of $B_{S}$ (L-shaped domain).}
		\label{fig:RD-tol-L}
	\end{center}
\end{figure}
 For subdomain $S$, the asymptotic behavior of $\widehat{E}_{L}$ with respect to  mesh size $h $ is shown in the Table \ref{table:result_L} and Figure \ref{fig:Errest_L}. 
  The numerical results tell that the local error component $E_1$ in $\widehat{E}_L$ gradually becomes dominant as the mesh is refined. 
\begin{table}[h]
\caption{Error estimators for subdomain $S$ (uniform mesh of L-shaped domain).}
\begin{center}
\begin{tabular}{|c|cc|c|cc|cc|}
\hline
\rule[-0.1cm]{0pt}{0.5cm}{} 
$h$&$\kappa_h$&$C(h)$&$E_{L}$&$E_1$&$E_2$&$\widehat{E}_L$&$\widehat{E}_{G}$\\\hline\hline
      1/16  & 0.046 & 0.050 & 0.080 & 0.139 & 0.108 & 0.192 & 0.172      \\
      1/32  & 0.028 & 0.029 & 0.050 & 0.081 & 0.047 & 0.098 & 0.095      \\
      1/64  & 0.017 & 0.018 & 0.032 & 0.049 & 0.021 & 0.054 & 0.055      \\
      1/128 &0.011 & 0.011 & 0.020 & 0.030 & 0.030  & 0.032 & 0.032  \\
      1/256 & 0.007 & 0.006 & 0.013 & 0.018 & 0.005  & 0.019 & 0.020  \\\hline
\end{tabular}
\label{table:result_L}
\end{center}
%}
\end{table}

\begin{figure}[h]
	\begin{center}
	\includegraphics[angle=0,width=5.6cm]{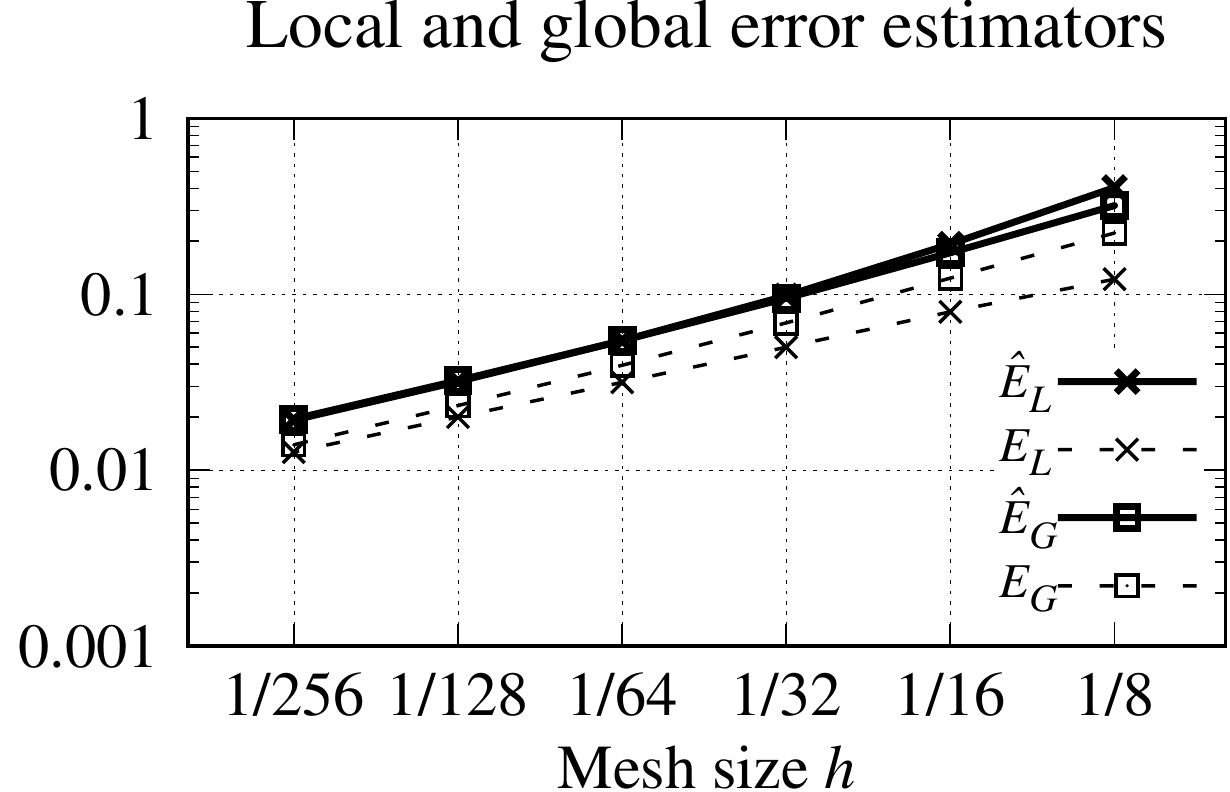} \hspace{5pt}
	\includegraphics[angle=0,width=5.6cm]{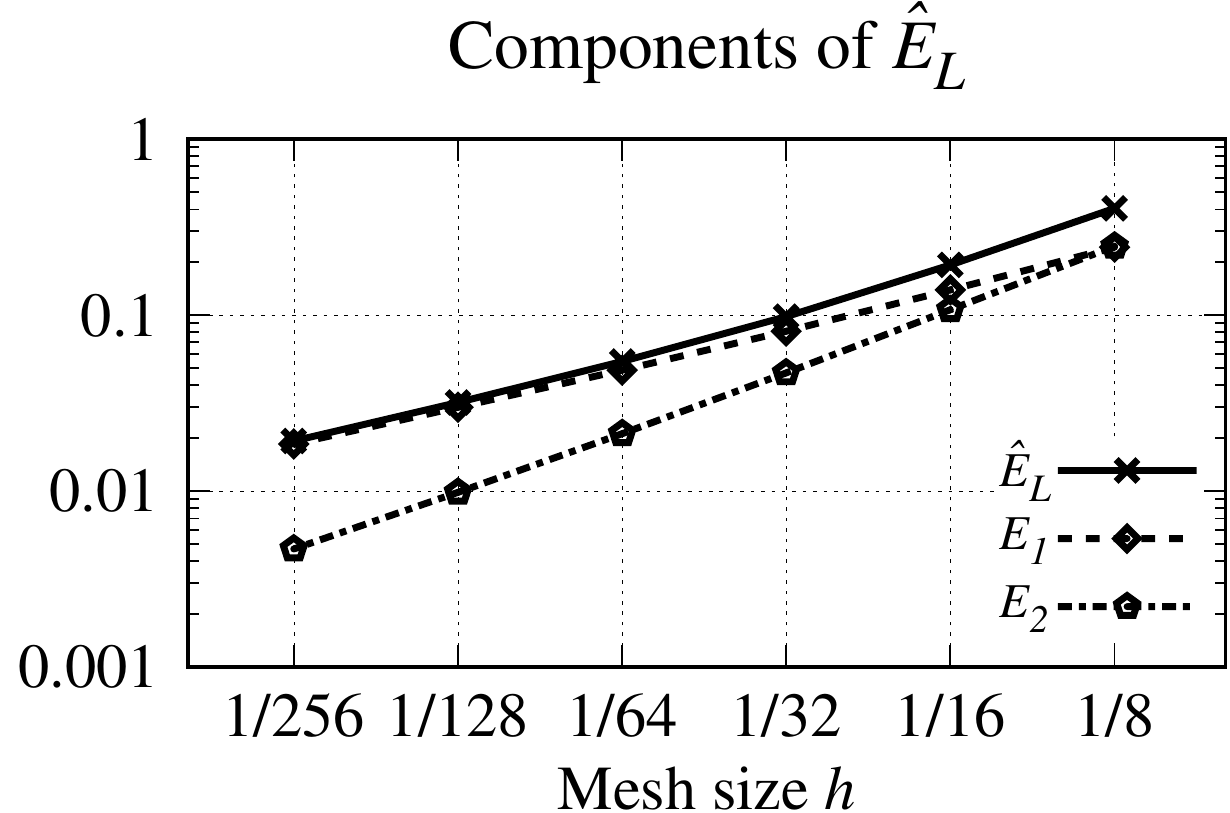} \\
	\vspace{10pt}
	\caption{Error estimators for subdomain $S$ (uniform mesh of L-shaped domain).}
	\label{fig:Errest_L}
	\end{center}
\end{figure}

  We also compare $E_L$, $\widehat{E}_L$ with $E_G$, $\widehat{E}_G$ in Table \ref{table:LGratio-S}. Denote 
  $\beta:=E_L/E_G$ and $\hat{\beta}:=\widehat{E}_L$/$\widehat{E}_G$. 
  It is observed that the approximation error concentrates in the subdomain $S$ around the re-entry corner as the mesh is refined. For $h = 1/256$, the local error in $S$ is about $91\%$ of the global error in the whole domain.
\begin{table}[h]
\caption{Comparison of the estimated local error on $S$ and $S'$.
\label{table:comparision-two-subdomains}}
\begin{center}
\begin{tabular}{c|cc|cc|cc|}
\cline{2-7} 
\rule[-0.05cm]{0pt}{0.4cm}{} 
		&\multicolumn{2}{|c|}{subdomain $S$} &\multicolumn{4}{|c|}{subdomain $S'$}  \\\hline
\multicolumn{1}{|c|}{\rule[-0.05cm]{0pt}{0.4cm}{} $h$}   &$\beta$(\%) &$\hat{\beta}$(\%) &$E_L$&$\widehat{E}_L$&$\beta$(\%) &$\hat{\beta}$(\%)  \\\hline\hline        
\multicolumn{1}{|c|}{1/16}    & 64 & 112   & 0.041 & 0.160 &	33  &	93            \\
\multicolumn{1}{|c|}{1/32}    & 73 & 103  &  0.021 & 0.069 &	30 & 72         \\
\multicolumn{1}{|c|}{1/64}     & 80 & 100   &0.010 & 0.031 &	26 &	57         \\
\multicolumn{1}{|c|}{1/128}   & 86 & 99    &	0.005 & 0.014 &	22 &	45           \\
\multicolumn{1}{|c|}{1/256}    & 91 & 99  & 0.003 & 0.007  &	18&	35       \\\hline
\end{tabular}
\label{table:LGratio-S}
\end{center}
\end{table}
Finally, we consider the local error estimation for a non-uniform mesh; see computation results Table \ref{table:result_L_nonuni} and Figure \ref{fig:Errest_L_nonuni}.
It is observed that for the subdomain $S$, both $\beta$ and $\hat{\beta}$ become smaller compared to the results in the case of uniform meshes, which implies that a denser mesh around the re-entry corner improves the quality of local approximation.

\begin{table}[h]
\caption{Error estimators for subdomain $S$ (non-uniform mesh of L-shaped domain)}
\begin{center}
\begin{tabular}{|c|cc|c|cc|cc|cc|}
\hline
\rule[-0.1cm]{0pt}{0.4cm}{} 
$h$&$\kappa_h$&$C(h)$&$E_{L}$&$E_1$&$E_2$&$\widehat{E}_L$&$\widehat{E}_{G}$&$\beta$(\%) &$\hat{\beta}$(\%) \\\hline\hline
      0.141    & 0.039 & 0.054   & 0.023   & 0.083 & 0.106 & 0.168   & 0.166   &	21 &	101        \\
      0.081  & 0.021 & 0.030  & 0.012 & 0.039 & 0.040 & 0.067 & 0.081   &	21 &	82       \\
      0.041  & 0.011 & 0.015  & 0.007 & 0.019 & 0.015 & 0.027 & 0.040    &	23 &	67      \\
      0.020  & 0.006 & 0.008 & 0.004 & 0.010 & 0.005 & 0.012  & 0.020   &	46 &	60        \\
      0.010 &0.003 & 0.004  & 0.002 & 0.005 & 0.002 & 0.006 & 0.010   &	30 &	57   \\\hline%& 0.002
\end{tabular}
\label{table:result_L_nonuni}
\end{center}
%}
\end{table}
\begin{figure}[h]
	\begin{center}
	\includegraphics[angle=0,width=5.6cm]{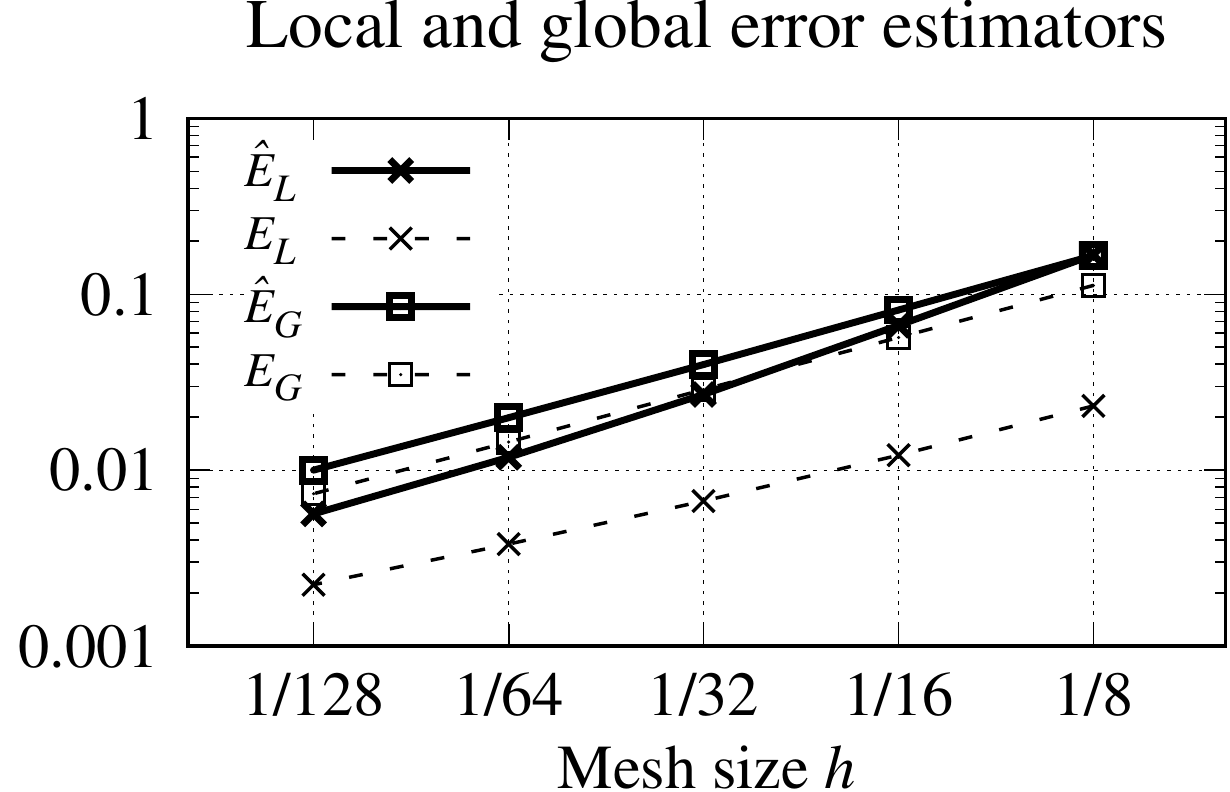} \hspace{5pt}
	\includegraphics[angle=0,width=5.6cm]{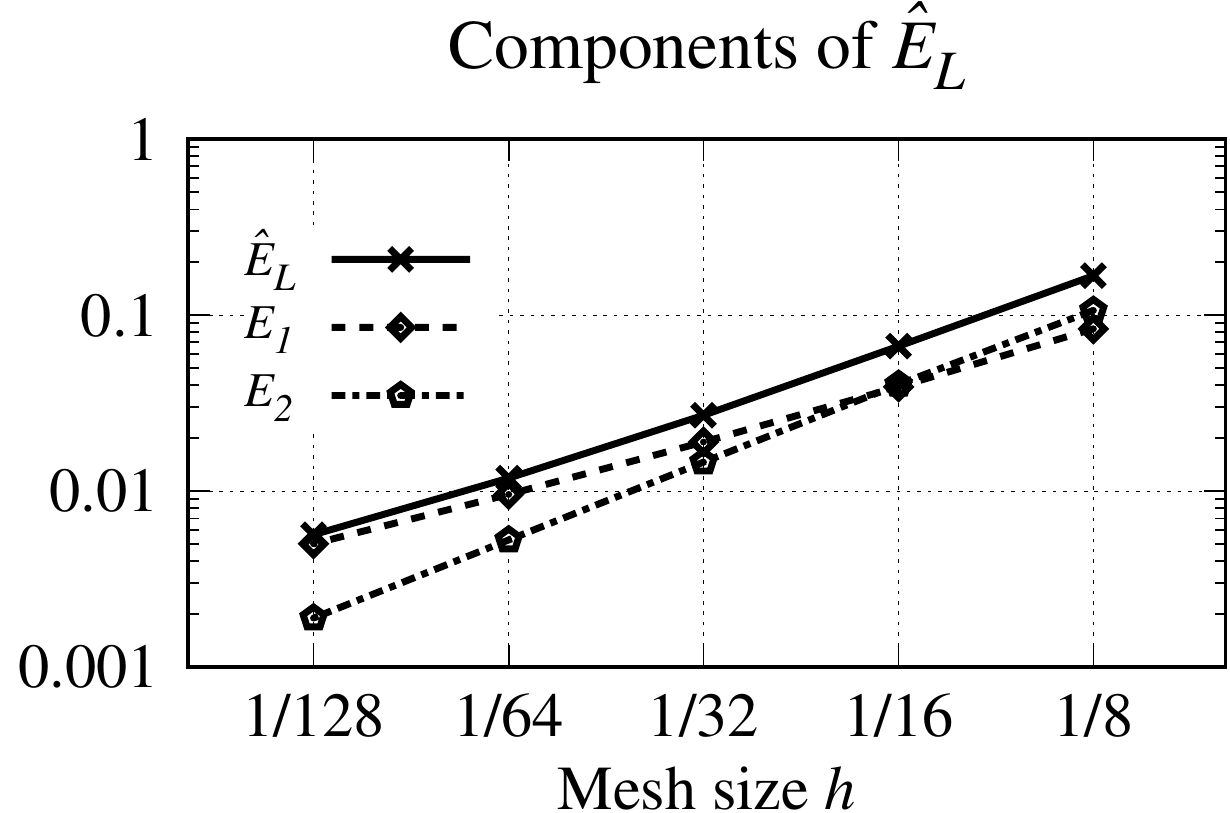} \\\vspace{20pt}
	\caption{Error estimators for subdomain $S$  (non-uniform mesh of L-shaped domain).}
	\label{fig:Errest_L_nonuni}
	\end{center}
\end{figure}

\section{Conclusion and future work}
\label{sec:7}

In this study, we proposed an explicit local {\em a posterior} error estimation for the finite element solutions and performed numerical experiments on the boundary value problem   of the Poisson equation, defined on the square and L-shaped domains.
The numerical results show that the proposed method provides an efficient estimation of the local error, especially compared to the general overestimated global error estimation.

In future, we will further apply the local error estimation to the four-probe method used in resistivity measurement. Another promising approach for the explicit point-wise error estimation needed by the four-probe method is to apply the idea of \cite{Fujita-1955} and the hypercircle method to the finite element method.

%\newpage

   \bibliographystyle{siamplain}
   \bibliography{MyBibliography.bib}

  \end{document}